\RequirePackage{fix-cm}
\documentclass[smallextended]{svjour3}       
\smartqed  
\usepackage{graphicx}
\usepackage{enumerate}
\usepackage{color}
\usepackage{amsfonts,amssymb} 
\usepackage{amsmath} 
\begin{document}

\title{A unified convergence analysis for the fractional diffusion equation driven by fractional Gaussion noise with Hurst index $H\in(0,1)$
	\thanks{This work was supported by the National Natural Science Foundation of China under
		Grant No. 12071195, and the AI and Big Data Funds under Grant No. 2019620005000775.
	}
}

\titlerunning{A Unified Convergence for Stochastic Fractional Diffusion Equation}        

\author{Daxin Nie         \and
        Weihua Deng 
}


\institute{Daxin Nie \at
              School of Mathematics and Statistics, Gansu Key Laboratory of Applied Mathematics and Complex Systems, Lanzhou University, Lanzhou 730000, P.R. China \\
              \email{ndx1993@163.com}           
           \and
           Weihua Deng \at
              School of Mathematics and Statistics, Gansu Key Laboratory of Applied Mathematics and Complex Systems, Lanzhou University, Lanzhou 730000, P.R. China\\
              \email{dengwh@lzu.edu.cn}
}

\date{Received: date / Accepted: date}

\maketitle

\begin{abstract}
Here, we provide a unified framework for numerical analysis of stochastic nonlinear fractional diffusion equation driven by fractional Gaussian noise with Hurst index $H\in(0,1)$.
A novel estimate of the second moment of the stochastic integral with respect to fractional Brownian motion is constructed, which greatly contributes to the regularity analyses of the solution in time and space for $H\in(0,1)$.
Then we use spectral Galerkin method and backward Euler convolution quadrature to discretize the fractional Laplacian and Riemann-Liouville fractional derivative, respectively. The sharp error estimates of the built numerical scheme are also obtained. Finally, the extensive numerical experiments verify the theoretical results.
\end{abstract}

\section{Introduction}

In this paper, we numerically solve the stochastic nonlinear fractional diffusion equation driven by fractional Gaussian noise with Hurst index $H\in(0,1)$, i.e.,
\begin{equation}\label{equretosol}
	\left \{
	\begin{split}
		&\partial_{t} u+\!_0\partial^{1-\alpha}_tA^{s} u
		=f(u)+\dot{W}^{H}_{Q} \,~\qquad\quad {\rm in}\ D,\ t\in(0,T],\\
		&u(\cdot,0)=0 \,\,\qquad\qquad\qquad\qquad \qquad\qquad{\rm in}\ D,\\
		&u=0 \qquad\qquad\qquad\qquad\qquad\qquad\quad \,\,\,\ \ \, {\rm on}\ \partial D,\ t\in(0,T],
	\end{split}
	\right .
\end{equation}
where  $D\subset \mathbb{R}^{d}$ $(d=1,2,3)$ is a convex polygonal domain; $A^{s}$ with $s\in(0,1)$ is fractional Laplacian defined by
\begin{equation*}
	A^{s}u=\sum_{k=1}^{\infty}\lambda_{k}^{s}(u,\phi_{k})\phi_{k},
\end{equation*}
and we let $A=-\Delta$ with a zero Dirichlet boundary condition, which has non-decreasing eigenvalues $\{\lambda_{k}\}_{k=1}^{\infty}$ and $L^{2}$-norm normalized eigenfunctions $\{\phi_{k}\}_{k=1}^{\infty}$; $f(u)$ is a nonlinear term and we assume that for $u,v\in \mathbb{H}$, there are
\begin{equation}\label{eqassptf}
	\begin{aligned}
		\|f(u)\|_{\mathbb{H}}\leq& C(1+\|u\|_{\mathbb{H}}),\\
		\|f(u)-f(v)\|_{\mathbb{H}}\leq& C\|u-v\|_{\mathbb{H}}
	\end{aligned}
\end{equation}
with $C$ being a positive constant; and ${}_0\partial^{1-\alpha}_t$ with $\alpha\in(0,1)$ is the Riemann-Liouville fractional derivative, whose definition is \cite{Podlubny.1999Fde}
\begin{equation}
	_{0}\partial^{1-\alpha}_tu=\frac{1}{\Gamma(\alpha)}\frac{\partial}{\partial t}\int^t_{0}(t-\xi)^{\alpha-1}u(\xi)d\xi.
\end{equation}
Here $W^{H}_{Q}$ with $H\in(0,1)$ is the fractional Gaussian process on a complete filtered probability space $(\Omega,\mathcal{F},\mathbb{P},\{\mathcal{F}_{t}\}_{t\geq 0})$ defined by
\begin{equation*}
	W^{H}_{Q}=\sum_{k=1}^{\infty}\sqrt{\Lambda_{k}}\phi_{k}W^{H}_{k},
\end{equation*}
where $W^{H}_{k}$ is one dimensional fBm and $Q$ is a self-adjoint, non-negative, linear operator and its eigenfunctions are same with $A$, and $\{\Lambda_{k}\}_{k=1}^{\infty}$ are the corresponding eigenvalues of $Q$. Then $A^{-\rho}Q^{1/2}$ with $\rho$ being a real number is a  Hilbert-Schmidt operator on $\mathbb{H}=L^{2}(D)$.

Now, we present a brief introduction to (\ref{equretosol}). Let $D$ be a bounded domain, $B(t)$ be a Brownian motion with $B(0)\in D$, and $\tau_D=\inf\{t>0: B(t)\notin D\}$. Denote $T_t$ as a $s$-stable subordinator. Let
\begin{equation*}
 X(t)=\left\{
\begin{array}{cc}
B(T_t),\quad &T_t<\tau_D,\\
\Theta,\quad &T_t\ge\tau_D
\end{array}
 \right.
 \end{equation*}
with $\Theta$ being a coffin state, implying to subordinate a killed Brownian motion (when first leaving the domain $D$).
The infinitesimal generator of $ X(t)$ is $A^{s}$ \cite{Song.2003Vondra}.  Further do the time change to $X(t)$ by the inverse $\alpha$-stable subordinator, the Fokker-Planck equation of which is $\partial_{t} u+\!_0\partial^{1-\alpha}_tA^{s} u =0 $, describing the competition between superdiffusion and subdiffusion. If the population of the particles is also being affected by the external source term depending on the density of the particles and the external fractional Gaussian noise, the Fokker-Planck equation Eq. (\ref{equretosol}) is reached. Most of the time, the influence of external noise  \cite{Mandelbrot.1968FBMFNaA,Simonsen.2003Macitnesmbw} on a system is unavoidable. Fractional Gaussian noise is one of the most popular external noises, the Hurst index $H$ of which describes the long-range dependence of the fractional Gaussian process \cite{Banna.2019Fbm}. To be specific, the fractional Gaussian process with $H\in(0,1/2)$ and $H\in(1/2,1)$ can be used to model the phenomena with a short memory and a long memory, respectively. From the viewpoint of mathematics, the Hurst index reflects the H\"older property of fractional Brownian motion's trajectory.


In the past few decades, there have been some numerical discussions for the stochastic PDEs driven by fractional Gaussian noise with the index $H\in(1/2,1)$ or $H=1/2$ \cite{Arezoomandan.2021ScmfspdewfBm,Li.2017GFEAfSSTFWE,Liu.2021HOAfSSFWEFbaASTGN,Wang.2017SmsrrfSwfnaocrftna,Wu.2020AaotLsfsspdbistwn,Yan.2019OeeffspdewfBm}. In addition,  \cite{Cao.2017ASEEwAWaRN,Cao.2018FeafsosdedbfBm} use the finite element method to solve the PDE driven by spatial fractional Gaussian noise with an index $H\in(0,1/2)$, where some special Green functions and It\^{o} isometry are used to provide the regularity of the solution, but these techniques can not reflect the influence of the temporal fractional Gaussian noise with $H\in(0,1/2)$ on the regularity of the mild solution and there are hardly researches for the temporal fractional Gaussian noise with $H\in(0,1/2)$. To try to fill the gap, a unified argument for $H\in(0,1)$ is proposed in this paper.  A key step of the analysis is to
give a novel estimate for $H\in(0,1)$  by the It\^{o} isometry and the equivalence of different fractional Sobolev spaces, i.e.,
\begin{equation}\label{eqintro}
	\mathbb{E} \left (\int_{0}^{T} g(r) d W^{H}(r)\right )^{2}\leq C\|{}_{0}\partial_{t}^{\frac{1-2H}{2}}g\|_{L^{2}([0,T])}^{2},
\end{equation}
which makes $\mathbb{E} \left (\int_{0}^{T} g(r) d W^{H}(r)\right )^{2}$ be bounded by a convolution of $g$ instead of the multiple integral of $g$. Thanks to \eqref{eqintro}, we obtain the sharp regularity estimates of the mild solution in time and space for \eqref{equretosol} with $H\in (0,1)$ by operator theory approach. Then the full discretization is built by the spectral Galerkin method in space and backward Euler convolution quadrature in time, respectively; the optimal error estimates are provided by transforming the solutions of discrete schemes into a convolution form and using regularity estimates of the mild solution and \eqref{eqintro}. It must be emphasized that, because of the nonlinear term $f(u)$ and the operator ${}_{0}\partial_{t}^{\frac{1-2H}{2}}$, the derivation of temporal error estimate is not an easy task and some new skills based on Laplace transform need to be introduced; for the details,  see Section 4. At the same time, the corresponding error estimates can show the influence of $H$, $s$, and $\alpha$ on convergence rates. Finally, the numerical examples validate the effectiveness of the numerical scheme. To the best of our knowledge, this is the first work on strong convergence analysis for the stochastic PDE driven by fractional Gaussian noise with $H\in(0,1)$.

The rest of the paper is organized as follows. In Section 2, we first provide some preliminaries and useful lemmas, and then study the spatial regularity and temporal regularity of the mild solution of Eq. \eqref{equretosol}. We apply spectral Galerkin method to discretize fractional Laplacian and  derive optimal error estimate for semidiscrete scheme in Section 3. In Section 4, the backward Euler convolution quadrature method  is used to discretize the time fractional derivative and the temporal error estimate of the numerical scheme is also obtained. We  provide some numerical examples to validate the  theoretically predicted convergence order in Section 5. The paper is concluded with some discussions in the last section. Throughout the paper, we denote by $C$ a generic positive constant, whose value may differ at different occurrences.

\section{Regularity of the solution}
\subsection{Preliminaries}

We first recall some definitions and properties on fractional Sobolev spaces \cite{Acosta.2019Feaffep,Acosta.2017AFLERoSaFEA,DiNezza.2012HgttfSs}. Let $D\subset \mathbb{R}^{d}$. The fractional Sobolev space $H^{s}(D)$ with $s\in(0,1)$ can be defined by
\begin{equation*}
	H^{s}(D)=\left \{w\in H^{s}(D):|w|_{H^{s}(D)}=\int_{D}\int_{D}\frac{(w(x)-w(y))^{2}}{|x-y|^{d+2s}}dxdy<\infty\right \},
\end{equation*}
whose norm can be written as $\|\cdot\|_{H^{s}(D)}=\|\cdot\|_{L^{2}(D)}+|\cdot|_{H^{s}(D)}$. The functions in $H^{s}(\mathbb{R}^{d})$ with support in $\bar{D}$ consist in
\begin{equation*}
	H^{s}_{0}(D)=\left \{w\in H^{s}(\mathbb{R}^{d}), {\rm\bf supp}~w\subset \bar{D}\right \},
\end{equation*}
with norm
\begin{equation*}
	\|w\|_{H^{s}_{0}(D)}=\|w\|_{L^{2}(D)}+c_{d,s}\int_{\mathbb{R}^{d}}\int_{\mathbb{R}^{d}}\frac{(w(x)-w(y))^{2}}{|x-y|^{d+2s}}dxdy,
\end{equation*}
which can also be expressed as
\begin{equation*}
	\|w\|_{H^{s}_{0}(D)}=\|(1+|\omega|^{2})^{s/2}\mathcal{F}(w)(\omega)\|_{L^{2}(\mathbb{R}^{d})},
\end{equation*}
with $\mathcal{F}(w)$ being the Fourier transform of $w$ and $c_{d,s}=\frac{2^{2s}s\Gamma(d/2+s)}{\pi^{d/2}\Gamma(1-s)}$.
\begin{remark}\label{reeqsob}
	According to \cite{Acosta.2019Feaffep},  $H^{s}(D)$ coincides with $H^{s}_{0}(D)$ when $s\in(0,\frac{1}{2})$.
\end{remark}

In what follows, we provide some preliminary facts on Hilbert-Schmidt operator, which can refer to \cite{Kloeden.1992Nsosde,Mishura.2008ScffBmarp}. Let $\mathcal{L}(\mathbb{U};\mathbb{V})$ be the Banach space consisting of all bounded linear operators $\mathbb{U}\rightarrow \mathbb{V}$, where $\mathbb{U}$ and $\mathbb{V}$ are two separable Hilbert spaces and their norms and inner products are denoted by $\|\cdot\|_{\mathbb{U}}$, $\|\cdot\|_{\mathbb{V}}$, and $(\cdot,\cdot)_{\mathbb{U}}$ and $(\cdot,\cdot)_{\mathbb{V}}$, respectively. Let $\mathcal{L}_{2}(\mathbb{U};\mathbb{V}) \,(\subset \mathcal{L}(\mathbb{U};\mathbb{V}))$ consist of all Hilbert-Schmidt operators, whose norm and inner product are defined by
\begin{equation*}
	\|T\|^{2}_{\mathcal{L}_{2}(\mathbb{U},\mathbb{V})}=\sum_{j\in \mathbb{N}^{*}}\|T\mu_{j}\|^{2}_{\mathbb{V}},~~~ \langle S,T \rangle_{\mathcal{L}_{2}(\mathbb{U},\mathbb{V})}=\sum_{j\in\mathbb{N}^{*}}( S\mu_{j},T\mu_{j})_{\mathbb{V}},~~ S,T\in \mathcal{L}_{2}(\mathbb{U},\mathbb{V}),
\end{equation*}
where $\{\mu_{j}\}_{j\in \mathbb{N}^{*}}$ are the orthonormal bases in $\mathbb{U}$ and the above definitions are independent of the specific choice of orthonormal bases.
Let $\mathbb{H}=L^{2}(D)$ with inner product $(\cdot,\cdot)$ and covariance operator $Q$ be a  self-adjoint, nonnegative linear operator on $\mathbb{H}$. Below we  use the notation `$\tilde{~}$' for taking Laplace transform; $\mathbb{E}$ denotes expectation; let $\epsilon > 0$ be arbitrarily small number and abbreviate $\|\cdot\|_{\mathcal{L}(\mathbb{H},\mathbb{H})}$ as $\|\cdot\|$.

Finally, we provide the definitions of the sectors and contour that will be used in the following proofs.
For $\kappa>0$ and $\pi/2<\theta<\pi$, we define sectors 
\begin{equation*}
	\begin{aligned}
		&\Sigma_{\theta}=\{z\in\mathbb{C}:z\neq 0,|\arg z|\leq \theta\},\ \Sigma_{\theta,\kappa}=\{z\in\mathbb{C}:|z|>\kappa,|\arg z|\leq \theta\},\\
	\end{aligned}
\end{equation*}
and 
the contour $\Gamma_{\theta,\kappa}$ by
\begin{equation*}
	\Gamma_{\theta,\kappa}=\{r e^{-\mathbf{i}\theta}: r\geq \kappa\}\cup\{\kappa e^{\mathbf{i}\psi}: |\psi|\leq \theta\}\cup\{r e^{\mathbf{i}\theta}: r\geq \kappa\},
\end{equation*}
where the circular arc is oriented counterclockwise and the two rays are oriented with an increasing imaginary part and $\mathbf{i}^2=-1$.

\subsection{Some useful lemmas}
Here, we provide a few technical lemmas, which can help to get the regularity of the solution of Eq. \eqref{equretosol}.
Let's first recall the definitions of left- and right-sided Riemann-Liouville fractional integrals and derivatives.
\begin{definition}[\cite{Podlubny.1999Fde}]
	The left- and right-sided Riemann-Liouville fractional integrals of order $\alpha$ $(\alpha>0)$ are defined by
	\begin{equation*}
		\begin{aligned}
			{}_{a}\partial^{-\alpha}_{x}u=\frac{1}{\Gamma(\alpha)}\int_{a}^{x}(x-\xi)^{\alpha-1}u(\xi)d\xi,\\
			{}_{x}\partial^{-\alpha}_{b}u=\frac{1}{\Gamma(\alpha)}\int_{x}^{b}(\xi-x)^{\alpha-1}u(\xi)d\xi
		\end{aligned}
	\end{equation*}
	with $a,b\in\mathbb{R}$. The Fourier transforms of ${}_{-\infty}\partial^{-\alpha}_{x}u$ and ${}_{x}\partial^{-\alpha}_{\infty}u$ are
	\begin{equation*}
		\begin{aligned}
			\mathcal{F}({}_{-\infty}\partial^{-\alpha}_{x}u)(\omega)=(\mathbf{i}\omega)^{-\alpha}\mathcal{F}(u)(\omega),\\
			\mathcal{F}({}_{x}\partial^{-\alpha}_{\infty}u)(\omega)=(-\mathbf{i}\omega)^{-\alpha}\mathcal{F}(u)(\omega).\\
		\end{aligned}
	\end{equation*}
\end{definition}

\begin{definition}[\cite{Podlubny.1999Fde}]
	The left- and right-sided Riemann-Liouville fractional derivatives of order $\alpha$ $(\alpha\in(0,1))$ are defined by
	\begin{equation*}
		\begin{aligned}
			{}_{a}\partial^{\alpha}_{x}u=\partial_{x}~_{a}\partial^{\alpha-1}_{x}u=\frac{1}{\Gamma(1-\alpha)}\frac{\partial}{\partial x}\int_{a}^{x}(x-\xi)^{-\alpha}u(\xi)d\xi,\\
			{}_{x}\partial^{\alpha}_{b}u=\partial_{x}~_{x}\partial^{\alpha-1}_{b}u=\frac{1}{\Gamma(1-\alpha)}\frac{\partial}{\partial x}\int_{x}^{b}(\xi-x)^{-\alpha}u(\xi)d\xi
		\end{aligned}
	\end{equation*}
	with $a,b\in\mathbb{R}$. The Fourier transforms of ${}_{-\infty}\partial^{\alpha}_{x}u$ and ${}_{x}\partial^{\alpha}_{\infty}u$ can be written as
	\begin{equation*}
		\begin{aligned}
			\mathcal{F}({}_{-\infty}\partial^{\alpha}_{x}u)(\omega)=(\mathbf{i}\omega)^{\alpha}\mathcal{F}(u)(\omega),\\
			\mathcal{F}({}_{x}\partial^{\alpha}_{\infty}u)(\omega)=(-\mathbf{i}\omega)^{\alpha}\mathcal{F}(u)(\omega).\\
		\end{aligned}
	\end{equation*}
\end{definition}
According to \cite{Ervin.2006Vfftsfade}, for $u\in H^{s}(a,b)$ with $a,b\in \mathbb{R}$ and ${\bf supp}~u\subset(a,b)$, there exist two positive constants $C_{1}$ and $C_{2}$ such that
\begin{equation}\label{eqequiHandJ}
	\begin{aligned}
		&C_{1}\|{}_{a}\partial^{s}_{x}u\|_{L^{2}(a,b)}\leq |u|_{H^{s}(a,b)}\leq C_{2}\|{}_{a}\partial^{s}_{x}u\|_{L^{2}(a,b)},\\
		&C_{1}\|{}_{x}\partial^{s}_{b}u\|_{L^{2}(a,b)}\leq |u|_{H^{s}(a,b)}\leq C_{2}\|{}_{x}\partial^{s}_{b}u\|_{L^{2}(a,b)}.
	\end{aligned}
\end{equation}

Then we provide a lemma on left- and right-sided Riemann-Liouville fractional integrals.
\begin{lemma}\label{lemIIiso}
	For $u\in L^{2}(\mathbb{R})$, $\nu\in(0,\frac{1}{2})$ and ${\bf supp}~ u\subset (a,b)$ with $a,b\in \mathbb{R}$, we have
	\begin{equation*}
		\int_{a}^{b}{}_{a}\partial^{-\nu}_{x}u~_{x}\partial^{-\nu}_{b}udx=\frac{1}{2\Gamma(2\nu)}\int_{a}^{b}\int_{a}^{b} u(\xi)u(\eta)|\eta-\xi|^{2\nu-1}d\eta  d\xi
	\end{equation*}
	and
	\begin{equation*}
		\begin{aligned}
			\int_{a}^{b}\int_{a}^{b} u(\xi)u(\eta)|\eta-\xi|^{2\nu-1}d\eta  d\xi\leq C\|{}_{a}\partial^{-\nu}_{x}u\|_{L^{2}(a,b)}^{2},\\
			\int_{a}^{b}\int_{a}^{b} u(\xi)u(\eta)|\eta-\xi|^{2\nu-1}d\eta  d\xi\leq C\|{}_{x}\partial^{-\nu}_{b}u\|_{L^{2}(a,b)}^{2}.
		\end{aligned}
	\end{equation*}
	
\end{lemma}
\begin{proof}
	For ${\bf supp}~u\subset (a,b)$, there holds
	\begin{equation*}
		\begin{aligned}
			&\int_{a}^{b}{}_{a}\partial^{-\nu}_{x}u~_{x}\partial^{-\nu}_{b}udx\\
			=&\frac{1}{(\Gamma(\nu))^{2}}\int_{a}^{b}\int_{a}^{x}(x-\xi)^{\nu-1}u(\xi)d\xi \int_{x}^{b}(\eta-x)^{\nu-1}u(\eta)d\eta dx\\
			=&\frac{1}{(\Gamma(\nu))^{2}}\int_{a}^{b}\int_{a}^{x}\int_{x}^{b}(x-\xi)^{\nu-1} (\eta-x)^{\nu-1}u(\xi)u(\eta)d\eta d\xi  dx\\
			=&\frac{1}{(\Gamma(\nu))^{2}}\int_{a}^{b}\int_{\xi}^{b}\int_{x}^{b}(x-\xi)^{\nu-1} (\eta-x)^{\nu-1}u(\xi)u(\eta)d\eta dx d\xi  \\
			=&\frac{1}{(\Gamma(\nu))^{2}}\int_{a}^{b}\int_{\xi}^{b}\int_{\xi}^{\eta}(x-\xi)^{\nu-1} (\eta-x)^{\nu-1}dx u(\xi)u(\eta)d\eta  d\xi  \\
			=&\frac{1}{\Gamma(2\nu)}\int_{a}^{b}\int_{\xi}^{b} u(\xi)u(\eta)(\eta-\xi)^{2\nu-1}d\eta  d\xi  \\
			=&\frac{1}{2\Gamma(2\nu)}\int_{a}^{b}\int_{a}^{b} u(\xi)u(\eta)|\eta-\xi|^{2\nu-1}d\eta  d\xi.
		\end{aligned}
	\end{equation*}
	For ${\rm\bf supp}~u\subset (a,b)$ and $\nu\in(0,\frac{1}{2})$, Parseval's equality leads to
	\begin{equation}\label{eqIIiso}
		\int_{a}^{b}{}_{a}\partial^{-\nu}_{x}u~{}_{x}\partial^{-\nu}_{b}udx=\int_{\mathbb{R}}{}_{-\infty}\partial^{-\nu}_{x}u~{}_{x}\partial^{-\nu}_{\infty}udx=\cos(\nu\pi)\||\omega|^{-\nu}\mathcal{F}(u)\|^{2}_{L^{2}(\mathbb{R})}.
	\end{equation}
	Moreover, for ${\bf supp}~u\subset (a,b)$ and $\nu\in(0,\frac{1}{2})$, by the Cauchy-Schwarz inequality, one has
	\begin{equation*}
		\begin{aligned}
			\int_{a}^{b}{}_{a}\partial^{-\nu}_{x}u~{}_{x}\partial^{-\nu}_{b}udx\leq& \frac{C}{\varepsilon_{1}}\|{}_{a}\partial^{-\nu}_{x}u\|_{L^{2}(a,b)}^{2}+\varepsilon_{1}\|{}_{x}\partial^{-\nu}_{b}u\|_{L^{2}(a,b)}^{2}\\
			\leq& \frac{C}{\varepsilon_{1}}\|{}_{a}\partial^{-\nu}_{x}u\|_{L^{2}(a,b)}^{2}+\varepsilon_{1}\|{}_{x}\partial^{-\nu}_{\infty}u\|_{L^{2}(\mathbb{R})}^{2}\\
			\leq& \frac{C}{\varepsilon_{1}}\|{}_{a}\partial^{-\nu}_{x}u\|_{L^{2}(a,b)}^{2}+\varepsilon_{1}\||\omega|^{-\nu}\mathcal{F}(u)\|_{L^{2}(\mathbb{R})}^{2}.
		\end{aligned}
	\end{equation*}
	Combining \eqref{eqIIiso} and taking a suitable $\varepsilon_{1}$, we can get the first desired result and the second one follows by similar arguments.
\end{proof}

For one-dimensional fBm, the following facts hold.
\begin{lemma}[\cite{Bardina.2006MfiwHplt12,Cao.2017ASEEwAWaRN,Cao.2018FeafsosdedbfBm}]\label{Lemitoeql05}
	For $H\in(0,1/2)$ and $g_{1}(t),g_{2}(t)\in H^{\frac{1-2H}{2}}_{0}([0,T])$, we have
	\begin{equation}\label{equisol05}
		\begin{aligned}
			&\mathbb{E} \left [\int_{0}^{T} g_{1}(r) d W^{H}(r) \int_{0}^{T} g_{2}(r) d W^{H}(r)\right ]\\
			&\qquad\qquad\qquad=\frac{1}{2}H(1-2H) \int_{\mathbb{R}} \int_{\mathbb{R}} \frac{(g_{1}(r_{1})-g_{1}(r_{2}))(g_{2}(r_{1})-g_{2}(r_{2}))}{|r_{1}-r_{2}|^{2 -2H}} d r_{1} d r_{2},
		\end{aligned}
	\end{equation}
	where $W^H$ means one-dimensional fBm.
\end{lemma}

\begin{lemma}[\cite{Kloeden.1992Nsosde,Mishura.2008ScffBmarp}]\label{Lemitoeqge05}
	For $H\in(1/2,1)$ and $g_{1}(t),g_{2}(t)\in L^2([0,T])$, there holds
	\begin{equation}\label{equisog05}
		\begin{aligned}
			&\mathbb{E} \left [\int_{0}^{T} g_{1}(r) d W^{H}(r) \int_{0}^{T} g_{2}(r) d W^{H}(r)\right ]\\
			&\qquad\qquad\qquad=H(2 H-1) \int_{0}^{T} \int_{0}^{T} g_{1}(r_{1}) g_{2}(r_{2})|r_{1}-r_{2}|^{2 H-2} d r_{1} d r_{2},
		\end{aligned}
	\end{equation}
	where $W^H$ means one-dimensional fBm.
\end{lemma}

Furthermore, we provide a lemma which plays a key role in this paper.

\begin{lemma}\label{eqcorleml}
	Let $g\in L^{2}([0,T])$,	${}_{0}\partial^{\frac{1-2H}{2}}_{t}g\in L^{2}([0,T])$ and $H\in(0,1)$. Then we have
	\begin{equation*}
		\begin{aligned}
			&\mathbb{E} \left (\int_{0}^{T} g(r) d W^{H}(r) \int_{0}^{T} g(r) d W^{H}(r)\right )\leq C\left \|{}_{0}\partial^{\frac{1-2H}{2}}_{t}g\right \|^{2}_{L^{2}([0,T])},\\
			&\mathbb{E} \left (\int_{0}^{T} g(T-r) d W^{H}(r) \int_{0}^{T} g(T-r) d W^{H}(r)\right )\leq C\left \|{}_{0}\partial^{\frac{1-2H}{2}}_{t}g\right \|^{2}_{L^{2}([0,T])}.
		\end{aligned}
	\end{equation*}
\end{lemma}
\begin{proof}
	When $H=\frac{1}{2}$, the desired results hold directly.
	As for $H\in(0,\frac{1}{2})$, using Lemma \ref{Lemitoeql05}, Eq. \eqref{eqequiHandJ}, and Remark \ref{reeqsob}, one has
	\begin{equation*}
		\begin{aligned}
			&\mathbb{E} \left (\int_{0}^{T} g(r) d W^{H}(r) \int_{0}^{T} g(r) d W^{H}(r)\right )\\
			=& \frac{1}{2}H(1-2H) \int_{\mathbb{R}} \int_{\mathbb{R}} \frac{(\bar{g}(r_{1})-\bar{g}(r_{2}))^{2}}{|r_{1}-r_{2}|^{2 -2H}} d r_{1} d r_{2}\\
			\leq& C\|g\|^{2}_{H^{\frac{1-2H}{2}}([0,T])}\\
			\leq &C\left \|{}_{0}\partial^{\frac{1-2H}{2}}_{t}g\right \|^{2}_{L^{2}([0,T])},
		\end{aligned}
	\end{equation*}
	where $\bar{g}$ is the zero extension of $g$. As for $H\in(\frac{1}{2},1)$,  Lemmas \ref{lemIIiso} and \ref{Lemitoeqge05} give
	\begin{equation*}
		\begin{aligned}
			&\mathbb{E} \left (\int_{0}^{T} g(r) d W^{H}(r) \int_{0}^{T} g(r) d W^{H}(r)\right )\\
			=& H(2H-1) \int_{0}^{T} \int_{0}^{T} \frac{g(r_{1})g(r_{2})}{|r_{1}-r_{2}|^{2 -2H}} d r_{1} d r_{2}\\
			\leq &C\left \|{}_{0}\partial^{\frac{1-2H}{2}}_{t}g\right \|^{2}_{L^{2}([0,T])}.
		\end{aligned}
	\end{equation*}
	As for the second estimate, it can be got similarly.
\end{proof}
\begin{remark}
	Different from the skills used in \cite{Arezoomandan.2021ScmfspdewfBm,Li.2017GFEAfSSTFWE,Liu.2021HOAfSSFWEFbaASTGN,Wang.2017SmsrrfSwfnaocrftna,Wu.2020AaotLsfsspdbistwn,Yan.2019OeeffspdewfBm}, with the help of Lemma \ref{eqcorleml}, we can
	obtain the corresponding regularity and error estimates for $H\in(0,1)$ by Laplace transform and operator theory approach.
\end{remark}
\begin{lemma}[\cite{Laptev.1997DaNEPoDiES,Li.1983OtSeatep}]\label{thmeigenvalue}
	Let $D$ be a bounded domain in $\mathbb{R}^d \,(d=1,2,3)$ and $\lambda_k$ the $k$-th eigenvalue of the Dirichlet boundary problem for the Laplace operator $A=-\Delta$ in $D$. Then, for all $k\geq 1$,
	\begin{equation*}
		\lambda_{k} \geq \frac{C_{d} d}{d+2} k^{2 / d}|D|^{-2 / d},
	\end{equation*}
	where $C_{d}=(2 \pi)^{2} B_{d}^{-2 / d}$, $|D|$ is the volume of $D$, and $B_d$ means the volume of the unit $d$-dimensional ball.
\end{lemma}

\subsection{A priori estimate of the solution}
In this subsection, we provide the regularity of the solution  of Eq. \eqref{equretosol} in time and space.
Firstly,
we denote operator $\mathcal{R}(t)$ as
\begin{equation*}
	\mathcal{R}(t)=\frac{1}{2\pi\mathbf{i}}\int_{\Gamma_{\theta,\kappa}}e^{zt}z^{\alpha-1}(z^{\alpha}+A^{s})^{-1}dz
\end{equation*}
and it satisfies
\begin{equation}\label{equresolvent0}
	\|A^{\beta s}\tilde{\mathcal{R}}(z)\|\leq C|z|^{\beta\alpha-1}\quad \forall z\in \Sigma_{\theta}~~{\rm and}~~\beta\in[0,1].
\end{equation}
Also, we can rewrite $\mathcal{R}(t)$ as
\begin{equation*}
	\mathcal{R}(t)u=\sum_{k=1}^{\infty}E_{k}(t)(u,\phi_{k})\phi_{k},
\end{equation*}
where $\{E_{k}(t)\}_{k=1}^{\infty}$ are defined by
\begin{equation*}
	E_{k}(t)=\frac{1}{2\pi\mathbf{i}}\int_{\Gamma_{\theta,\kappa}}e^{zt}z^{\alpha-1}(z^{\alpha}+\lambda_{k}^{s})^{-1}dz,
\end{equation*}
and the following estimate holds
\begin{equation}\label{equresolvent}
	|\lambda_{k}^{ s\beta}\tilde{E}_{k}(z)|\leq C|z|^{\beta\alpha-1}\quad \forall z\in \Sigma_{\theta}\quad{\rm and }\quad \beta\in [0,1].
\end{equation}

Thus by taking Laplace and inverse Laplace transforms, the mild solution of Eq. \eqref{equretosol} can be written as
\begin{equation}\label{eqrepsol}
	\begin{aligned}
		u=&\int_{0}^{t}\mathcal{R}(t-r)f(u(r))dr+\int_{0}^{t}\mathcal{R}(t-r)dW^{H}_{Q}(r)\\
		=&\int_{0}^{t}\mathcal{R}(t-r)f(u(r))dr+\sum_{k=1}^{\infty}\int_{0}^{t}\sqrt{\Lambda_{k}}E_{k}(t-r)\phi_{k}dW^{H}_{k}(r).
	\end{aligned}
\end{equation}

Then we can obtain the following spatial regularity of $u$.
\begin{theorem}\label{thmsobo}
	Let $u$ be the mild solution of Eq. $\eqref{equretosol}$ and $\|A^{-\rho}Q^{1/2}\|_{\mathcal{L}_{2}}<\infty$ with $\rho\in\left[0,\min(\frac{sH}{\alpha},s+\epsilon)\right)$ and $\alpha\in(0,1)$. Then the mild solution $u$ satisfies
	\begin{equation*}
		\mathbb{E}\|A^{\sigma }u\|^{2}_{\mathbb{H}}\leq C,
	\end{equation*}
	where $\sigma\in \left[0,\min(s-\rho,\frac{sH}{\alpha}-\rho-\epsilon)\right]$.
\end{theorem}
\begin{proof}
	Simple calculations give
	\begin{equation*}
		\begin{aligned}
			\mathbb{E}\|A^{\sigma }u\|^{2}_{\mathbb{H}}\leq& 2 \mathbb{E}\left \|\int_{0}^{t}A^{\sigma}\mathcal{R}(t-r)f(u(r))dr\right \|_{\mathbb{H}}^{2}+2 \mathbb{E}\left \|\int_{0}^{t}A^{\sigma}\mathcal{R}(t-r)dW^{H}_{Q}(r)\right \|_{\mathbb{H}}^{2}\\
			\leq&\uppercase\expandafter{\romannumeral1}+\uppercase\expandafter{\romannumeral2}.
		\end{aligned}
	\end{equation*}
	
	As for $\uppercase\expandafter{\romannumeral1}$, the resolvent estimate \eqref{equresolvent0}, assumption \eqref{eqassptf}, and the Cauchy-Schwarz inequality lead to
	\begin{equation*}
		\begin{aligned}
			\uppercase\expandafter{\romannumeral1}\leq& C\mathbb{E}\left (\int_{0}^{t}(t-r)^{-\sigma\alpha/s}(1+\|u(r)\|_{\mathbb{H}})dr\right )^{2}\\
			\leq & C\mathbb{E}\left (\int_{0}^{t}(t-r)^{-2\sigma\alpha/s+1-\epsilon}(1+\|u(r)\|_{\mathbb{H}}^{2})dr\right) \\
			\leq& C+C\int_{0}^{t}(t-r)^{-2\sigma\alpha/s+1-\epsilon}\mathbb{E}\|u(r)\|_{\mathbb{H}}^{2}dr,
		\end{aligned}
	\end{equation*}
	where we need to require $-2\sigma\alpha/s+1>-1$, i.e., $\sigma<\frac{s}{\alpha}$.
	
	As for $\uppercase\expandafter{\romannumeral2}$, using \eqref{eqrepsol}, Lemma \ref{eqcorleml}, and $\|A^{-\rho}Q^{1/2}\|_{\mathcal{L}_{2}}<\infty$, we have
	\begin{equation*}
		\begin{aligned}
			\uppercase\expandafter{\romannumeral2} &
\\
\leq& C\sum_{k=1}^{\infty}\mathbb{E}\left (\int_{0}^{t}\lambda_{k}^{\sigma }\sqrt{\Lambda_{k}}E_{k}(t-r)\phi_{k}dW^{H}_{k}(r),\int_{0}^{t}\lambda_{k}^{\sigma }\sqrt{\Lambda_{k}}E_{k}(t-r)\phi_{k}dW^{H}_{k}(r)\right )\\
			\leq&C\sum_{k=1}^{\infty}\int_{0}^{t}|\lambda_{k}^{\sigma }\sqrt{\Lambda_{k}}{~}_{0}\partial_{r}^{\frac{1-2H}{2}}E_{k}(r)|^{2}dr\\
			\leq& C\sup_{k\in \mathbb{N^{*}}}\int_{0}^{t}|\lambda_{k}^{\sigma +\rho}{~}_{0}\partial_{r}^{\frac{1-2H}{2}}E_{k}(r)|^{2}dr.
		\end{aligned}
	\end{equation*}
	By the definition of $E_{k}(r)$ and resolvent estimate \eqref{equresolvent}, one can get
	\begin{equation*}
		\begin{aligned}
			&\int_{0}^{t}|\lambda_{k}^{\sigma+\rho}{~}_{0}\partial_{r}^{\frac{1-2H}{2}}E_{k}(r)|^{2}dr\\
			\leq &C\int_{0}^{t}\left |\int_{\Gamma_{\theta,\kappa}}e^{zr}\lambda_{k}^{\sigma+\rho}z^{\frac{1-2H}{2}}\tilde{E}_{k}(z)dz\right |^{2}dr\\
			\leq &C\int_{0}^{t}\left (\int_{\Gamma_{\theta,\kappa}}|e^{zr}||z|^{\frac{1-2H}{2}+(\sigma+\rho)\alpha/s-1}|dz|\right )^{2}dr\\
			\leq& C\int_{0}^{t} r^{2H-1-2(\sigma+\rho)\alpha/s}dr.
		\end{aligned}
	\end{equation*}
	To preserve the boundedness of $\mathbb{E}\|A^{\sigma}u\|^{2}_{\mathbb{H}}$, we require $2H-1-2(\sigma+\rho)\alpha/s>-1$, which leads to $\sigma<\frac{sH}{\alpha}-\rho$. Combining the estimates of $\uppercase\expandafter{\romannumeral1}$ and $\uppercase\expandafter{\romannumeral2}$ and the Gr\"onwall inequality, the desired result is obtained.
\end{proof}

In the following, we provide the H\"{o}lder regularity of the mild solution $u$.
\begin{theorem}\label{thmholder}
	Let $u$ be the mild solution of Eq. $\eqref{equretosol}$ and $\|A^{-\rho}Q^{1/2}\|_{\mathcal{L}_{2}}<\infty$ with $\rho\in \left[0,\frac{sH}{\alpha}\right)\cap[0,s]$ and $\alpha\in(0,1)$. Then we have
	\begin{equation*}
		\mathbb{E}\left \|\frac{u(t)-u(t-\tau)}{\tau^{\gamma}}\right \|^{2}_{\mathbb{H}}\leq C,
	\end{equation*}
	where $\gamma\in \left[0,H-\frac{\rho\alpha}{s}\right)$.
\end{theorem}
\begin{proof}
	We first divide $\mathbb{E}\left \|\frac{u(t)-u(t-\tau)}{\tau^{\gamma}}\right \|_{\mathbb{H}}^{2}$ into two parts
	\begin{equation*}
		\begin{aligned}
			&\mathbb{E}\left \|\frac{u(t)-u(t-\tau)}{\tau^{\gamma}}\right \|_{\mathbb{H}}^{2}\\
			\leq&\mathbb{E}\left \|\frac{\int_{0}^{t}\mathcal{R}(t-r)f(u(r))dr-\int_{0}^{t-\tau}\mathcal{R}(t-\tau-r)f(u(r))dr}{\tau^{\gamma}}\right \|_{\mathbb{H}}^{2}\\
			&+\mathbb{E}\left \|\frac{\int_{0}^{t}\mathcal{R}(t-r)dW^{H}_{Q}(r)-\int_{0}^{t-\tau}\mathcal{R}(t-\tau-r)dW^{H}_{Q}(r)}{\tau^{\gamma}}\right \|_{\mathbb{H}}^{2}\\
			\leq& \uppercase\expandafter{\romannumeral1}+\uppercase\expandafter{\romannumeral2}.
		\end{aligned}
	\end{equation*}

	As for $\uppercase\expandafter{\romannumeral1}$, we have
	\begin{equation*}
		\begin{aligned}
			\uppercase\expandafter{\romannumeral1}\leq& \mathbb{E}\left \|\frac{\int_{0}^{t-\tau}(\mathcal{R}(t-r)-\mathcal{R}(t-r-\tau))f(u(r))dr}{\tau^{\gamma}}\right \|_{\mathbb{H}}^{2}\\
			&+\mathbb{E}\left \|\frac{\int_{t-\tau}^{t}\mathcal{R}(t-r)f(u(r))dr}{\tau^{\gamma}}\right \|_{\mathbb{H}}^{2}\\
			\leq&\uppercase\expandafter{\romannumeral1}_{1}+\uppercase\expandafter{\romannumeral1}_{2}.
		\end{aligned}
	\end{equation*}
	For $\uppercase\expandafter{\romannumeral1}_{1}$, the fact $|\frac{e^{z\tau}-1}{\tau^{\gamma}}|\leq C|z|^{\gamma}$ with $z\in\Gamma_{\theta,\kappa}$ and $\gamma\in[0,1]$ \cite{Gunzburger.2018ScrotdfstfPstastwn} and Eq. \eqref{eqassptf} give
	\begin{equation*}
		\begin{aligned}
			\uppercase\expandafter{\romannumeral1}_{1}\leq &C\mathbb{E}\left \|\int_{0}^{t-\tau}\int_{\Gamma_{\theta, \kappa}}\frac{e^{z(t-r)}-e^{z(t-r-\tau)}}{\tau^{\gamma}}\tilde{\mathcal{R}}(z)dzf(u(r))dr\right \|_{\mathbb{H}}^{2}\\
			\leq &C\mathbb{E}\left (\int_{0}^{t-\tau}\int_{\Gamma_{\theta, \kappa}}|e^{z(t-r-\tau)}||z|^{\gamma-1}|dz|\|f(u(r))\|_{\mathbb{H}}dr\right) ^{2}\\
			\leq &C\mathbb{E}\left (\int_{0}^{t-\tau}(t-\tau-r)^{-\gamma}(1+\|u(r)\|_{\mathbb{H}})dr\right) ^{2}\\
			\leq& C,
		\end{aligned}
	\end{equation*}
	where we need to require $\gamma\in[0,1)$. Similarly,
	\begin{equation*}
		\begin{aligned}
			\uppercase\expandafter{\romannumeral1}_{2}\leq&C\mathbb{E}\left(\frac{\int_{t-\tau}^{t} \|f(u(r))\|_{\mathbb{H}}dr}{\tau^{\gamma}}\right) ^{2}\\
			\leq&C\tau^{1-2\gamma}\int_{t-\tau}^{t} \mathbb{E}\|f(u(r))\|_{\mathbb{H}}^{2}dr\\
			\leq& C,
		\end{aligned}
	\end{equation*}
	where we need to require $\gamma\in[0,1]$.
	
	As for $\uppercase\expandafter{\romannumeral2}$, there are	
	\begin{equation*}
		\begin{aligned}
			&\uppercase\expandafter{\romannumeral2}\\
			\leq& C\mathbb{E}\left \|\frac{\sum_{k=1}^{\infty}\left (\int_{0}^{t}\sqrt{\Lambda_{k}}E_{k}(t-r)\phi_{k}dW^{H}_{k}(r)-\int_{0}^{t-\tau}\sqrt{\Lambda_{k}}E_{k}(t-\tau-r)\phi_{k}dW^{H}_{k}(r)\right )}{\tau^{\gamma}}\right \|_{\mathbb{H}}^{2}\\
			\leq& C\mathbb{E}\left \|\frac{\sum_{k=1}^{\infty}\int_{0}^{t-\tau}\sqrt{\Lambda_{k}}(E_{k}(t-r)-E_{k}(t-\tau-r))\phi_{k}dW^{H}_{k}(r)}{\tau^{\gamma}}\right \|_{\mathbb{H}}^{2}\\
			&+C\mathbb{E}\left\|\frac{\sum_{k=1}^{\infty}\int_{t-\tau}^{t}\sqrt{\Lambda_{k}}E_{k}(t-r)\phi_{k}dW^{H}_{k}(r)}{\tau^{\gamma}}\right\|^{2}_{\mathbb{H}}\\
			\leq& \uppercase\expandafter{\romannumeral2}_{1}+\uppercase\expandafter{\romannumeral2}_{2}.
		\end{aligned}
	\end{equation*}
	
	For $\uppercase\expandafter{\romannumeral2}_{1}$, $\|A^{-\rho}Q^{1/2}\|<\infty$ and the definition of $E_{k}$ lead to
	\begin{equation*}
		\begin{aligned}
			\uppercase\expandafter{\romannumeral2}_{1}\leq& C\sum_{k=1}^{\infty}\int_{0}^{t-\tau}\left |\frac{\sqrt{\Lambda_{k}}{~}_{0}\partial^{\frac{1-2H}{2}}_{r}(E_{k}(r+\tau)-E_{k}(r))}{\tau^{\gamma}}\right |^{2}dr\\
			\leq& C\sup_{k\in\mathbb{N}^{*}}\int_{0}^{t-\tau}\left |\int_{\Gamma_{\theta,\kappa}}e^{zr}\frac{e^{z\tau}-1}{\tau^{\gamma}}\lambda_{k}^{\rho}z^{\frac{1-2H}{2}}\tilde{E}_{k}(z)dz\right |^{2}dr.
		\end{aligned}
	\end{equation*}
	Combining \eqref{equresolvent} and the fact $|\frac{e^{z\tau}-1}{\tau^{\gamma}}|\leq C|z|^{\gamma}$ with $z\in\Gamma_{\theta,\kappa}$ and $\gamma\in[0,1]$, one has
	\begin{equation*}
		\begin{aligned}
			&\int_{0}^{t-\tau}\left |\int_{\Gamma_{\theta,\kappa}}e^{zr}\frac{e^{z\tau}-1}{\tau^{\gamma}}\lambda_{k}^{\rho}z^{\frac{1-2H}{2}}\tilde{E}_{k}(z)dz\right |^{2}dr\\
			\leq &C\int_{0}^{t}\left (\int_{\Gamma_{\theta,\kappa}}|e^{zr}||z|^{\frac{1-2H}{2}+\gamma+\rho\alpha/s-1}|dz|\right )^{2}dr\\
			\leq& C\int_{0}^{t} r^{2H-1-2\gamma-2\rho\alpha/s}dr.
		\end{aligned}
	\end{equation*}
	To preserve the boundedness of $\uppercase\expandafter{\romannumeral2}_{1}$, $\gamma$ should satisfy $2H-1-2\gamma-2\rho\alpha/s>-1$, which yields $\gamma<H-\frac{\rho\alpha}{s}$.
	
	As for $\uppercase\expandafter{\romannumeral2}_{2}$, when $H=1/2$, the It{\^{o}} isometry shows
	\begin{equation*}
		\begin{aligned}
			\uppercase\expandafter{\romannumeral2}_{2}\leq& C\tau^{-2\gamma}\sum_{k=1}^{\infty}\int_{t-\tau}^{t}|\sqrt{\Lambda_{k}}E_{k}(t-r)|^{2}dr\\
			\leq & C\tau^{-2\gamma}\sum_{k=1}^{\infty}\int_{0}^{\tau}|\sqrt{\Lambda_{k}}E_{k}(r)|^{2}dr;
		\end{aligned}
	\end{equation*}
	when $H\in(1/2,1)$, by Lemmas \ref{lemIIiso} and \ref{Lemitoeqge05}, we deduce
	\begin{equation*}
		\begin{aligned}
			\uppercase\expandafter{\romannumeral2}_{2}\leq& C\tau^{-2\gamma}\sum_{k=1}^{\infty}\int_{t-\tau}^{t}\int_{t-\tau}^{t}\frac{\sqrt{\Lambda_{k}}E_{k}(t-r_{1})\sqrt{\Lambda_{k}}E_{k}(t-r_{2})}{|r_{1}-r_{2}|^{2-2H}}dr_{1}dr_{2}\\
			\leq& C\tau^{-2\gamma}\sum_{k=1}^{\infty}\int_{0}^{\tau}\int_{0}^{\tau}\frac{\sqrt{\Lambda_{k}}E_{k}(r_{1})\sqrt{\Lambda_{k}}E_{k}(r_{2})}{|r_{1}-r_{2}|^{2-2H}}dr_{1}dr_{2}\\
			\leq & C\tau^{-2\gamma}\sum_{k=1}^{\infty}\int_{0}^{\tau}|\sqrt{\Lambda_{k}}{~}_{0}\partial^{\frac{1-2H}{2}}_{r}E_{k}(r)|^{2}dr;
		\end{aligned}
	\end{equation*}
	when $H\in(0,1/2)$, by Lemma \ref{Lemitoeql05} and Remark \ref{reeqsob}, one can get
	\begin{equation*}
		\begin{aligned}
			&\uppercase\expandafter{\romannumeral2}_{2}\\ \leq& C\tau^{-2\gamma}\\
			&\,\cdot\sum_{k=1}^{\infty}\int_{\mathbb{R}}\int_{\mathbb{R}}\frac{(\sqrt{\Lambda_{k}}E_{k}(t-r_{1})\chi_{[t-\tau,t]}(r_{1})-\sqrt{\Lambda_{k}}E_{k}(t-r_{2})\chi_{[t-\tau,t]}(r_{2}))^{2}}{|r_{1}-r_{2}|^{2-2H}}dr_{1}dr_{2}\\
			\leq& C\tau^{-2\gamma}\sum_{k=1}^{\infty}\int_{0}^{\tau}\int_{0}^{\tau}\frac{(\sqrt{\Lambda_{k}}E_{k}(r_{1})-\sqrt{\Lambda_{k}}E_{k}(r_{2}))^{2}}{|r_{1}-r_{2}|^{2-2H}}dr_{1}dr_{2}\\
			\leq & C\tau^{-2\gamma}\sum_{k=1}^{\infty}\int_{0}^{\tau}|\sqrt{\Lambda_{k}}{~}_{0}\partial^{\frac{1-2H}{2}}_{r}E_{k}(r)|^{2}dr,
		\end{aligned}
	\end{equation*}
	where $\chi_{[a,b]}$ is the characteristic function on $[a,b]$.
	According to $\|A^{-\rho}Q^{1/2}\|_{\mathcal{L}_{2}}<\infty$, we have
	\begin{equation*}
		\begin{aligned}
			\uppercase\expandafter{\romannumeral2}_{2}
			\leq &C\tau^{-2\gamma}\sum_{k=1}^{\infty}\int_{0}^{\tau}|\sqrt{\Lambda_{k}}{~}_{0}\partial^{\frac{1-2H}{2}}_{r}E_{k}(r)|^{2}dr\\
			\leq &C\tau^{-2\gamma}\sup_{k\in\mathbb{N}^{*}}\int_{0}^{\tau}|\lambda_{k}^{\rho}{~}_{0}\partial^{\frac{1-2H}{2}}_{r}E_{k}(r)|^{2}dr.\\
		\end{aligned}
	\end{equation*}
	By \eqref{equresolvent}, one gets
	\begin{equation*}
		\begin{aligned}
			&\int_{0}^{\tau}|\lambda_{k}^{\rho}{}_{0}\partial^{\frac{1-2H}{2}}_{r}E_{k}(r)|^{2}dr\\
			\leq& C\int_{0}^{\tau}\left(\int_{\Gamma_{\theta,\kappa}}e^{zr}\lambda_{k}^{\rho}z^{\frac{1-2H}{2}}\tilde{E}_{k}(z)dz\right)^{2}dr\\
			\leq& C\int_{0}^{\tau}\left(\int_{\Gamma_{\theta,\kappa}}|e^{zr}||z|^{\frac{1-2H}{2}+\rho\alpha/s-1}|dz|\right)^{2}dr\\
			\leq& C\int_{0}^{\tau}r^{2H-1-2\rho\alpha/s}dr\\
			\leq& C\tau^{2H-2\rho\alpha/s},
		\end{aligned}
	\end{equation*}
	which implies $\gamma<H-\rho\alpha/s$. Thus the proof is completed.
\end{proof}

\section{Spatial discretization and error analysis}

In this section, we use a spectral Galerkin method to discretize fractional Laplacian and the error estimate is built. We first introduce a finite dimensional subspace of $\mathbb{H}$ by $\mathbb{H}_{N}={\rm span}\{\phi_{1},\phi_{2},\ldots,\phi_{N}\}$ with $N\in \mathbb{N}^{*}$ and define the projection operator $P_{N}:\mathbb{H}\rightarrow \mathbb{H}_{N}$ by
\begin{equation*}
	P_{N}u=\sum_{i=1}^{N}(u,\phi_{i})\phi_{i}\quad \forall u\in \mathbb{H}.
\end{equation*}
Define $A_{N}^{s}:\mathbb{H}_{N}\rightarrow \mathbb{H}_{N}$ as
\begin{equation*}
	(A_{N}^{s}u_{N},v_{N})=(A^{s}u_{N},v_{N})\quad \forall u_{N},v_{N}\in\mathbb{H}_{N}.
\end{equation*}
It is easy to verify that
\begin{equation*}
	A_{N}^{s}u_{N}=A_{N}^{s}P_{N}u_{N}=P_{N}A_{N}^{s}u_{N}=\sum_{k=1}^{N}\lambda_{k}^{s}(u_{N},\phi_{k})\phi_{k}.
\end{equation*}
Thus the spectral Galerkin semidiscrete scheme of (\ref{equretosol}) can be written as: find $u_{N}(t)\in\mathbb{H}_{N}$ satisfying
\begin{equation}\label{eqsemischeme}
	\left \{
	\begin{split}
		&\partial_{t} u_{N}+\!_0\partial^{1-\alpha}_tA^{s}_{N} u_{N}
		=P_{N}f(u_{N})+P_{N}\dot{W}^{H}_{Q} \,~\qquad\quad \ t\in(0,T],\\
		&u_{N}(0)=0.
	\end{split}
	\right .
\end{equation}
Taking Laplace transform and inverse Laplace transform gives
\begin{equation*}
	u_{N}=\int_{0}^{t}\mathcal{R}_{N}(t-r)P_{N}f(u_{N}(r))dr+\int_{0}^{t}\mathcal{R}_{N}(t-r)P_{N}dW^{H}_{Q}(r),
\end{equation*}
where
\begin{equation*}
	\mathcal{R}_{N}(t)=\frac{1}{2\pi\mathbf{i}}\int_{\Gamma_{\theta,\kappa}}e^{zt}z^{\alpha-1}(z^{\alpha}+A_{N}^{s})^{-1}dz.
\end{equation*}
By the definitions of $A^{s}_{N}$ and $P_{N}$, we rewrite $u_{N}$ as
\begin{equation}\label{eqsemiresol}
	u_{N}=\int_{0}^{t}\mathcal{R}_{N}(t-r)P_{N}f(u_{N}(r))dr+\sum_{k=1}^{N}\int_{0}^{t}\sqrt{\Lambda_{k}}E_{k}(t-r)\phi_{k}dW^{H}_{k}(r).
\end{equation}

Similar to the proofs of Theorems \ref{thmsobo} and \ref{thmholder}, one can get the following two estimates on $u_{N}$.
\begin{theorem}\label{thmsoboun}
	Let $u_{N}$ be the mild solution of Eq. $\eqref{eqsemischeme}$ and $\|A^{-\rho}Q^{1/2}\|_{\mathcal{L}_{2}}<\infty$ with $\rho\in\left[0,\min(\frac{sH}{\alpha},s+\epsilon)\right)$ and $\alpha\in(0,1)$. Then the mild solution $u_{N}$ satisfies
	\begin{equation*}
		\mathbb{E}\|A^{\sigma }u_{N}\|^{2}_{\mathbb{H}}\leq C,
	\end{equation*}
	where $\sigma\in \left[0,\min(s-\rho,\frac{sH}{\alpha}-\rho-\epsilon)\right]$.
\end{theorem}

\begin{theorem}\label{thmholderun}
	Let $u_{N}$ be the mild solution of Eq. $\eqref{eqsemischeme}$ and $\|A^{-\rho}Q^{1/2}\|_{\mathcal{L}_{2}}<\infty$ with $\rho\in [0,\frac{sH}{\alpha})\cap[0,s]$ and $\alpha\in(0,1)$. Then we have
	\begin{equation*}
		\mathbb{E}\left \|\frac{u_{N}(t)-u_{N}(t-\tau)}{\tau^{\gamma}}\right \|^{2}_{\mathbb{H}}\leq C,
	\end{equation*}
	where $\gamma\in[0,H-\frac{\rho\alpha}{s})$.
\end{theorem}

Now we provide the spatial error estimate.

\begin{theorem}\label{thmspter}
	Let $u$ and $u_{N}$ be the solutions of \eqref{equretosol} and \eqref{eqsemischeme}, respectively. Assuming $\|A^{-\rho}Q^{1/2}\|_{\mathcal{L}_{2}}<\infty$ with $\rho\in \left[0,\min(\frac{sH}{\alpha},s+\epsilon)\right)$ and $\alpha\in(0,1)$, we have
	\begin{equation*}
		(\mathbb{E}\|u-u_{N}\|_{\mathbb{H}}^{2})^{1/2}\leq C(N+1)^{-2\sigma/d},
	\end{equation*}
	where $\sigma\in\left(0,\min(s-\rho,\frac{sH}{\alpha}-\rho-\epsilon)\right]$ and $d$ is the dimension of $\Omega$.
\end{theorem}
\begin{proof}
	By Eqs. \eqref{eqrepsol} and \eqref{eqsemiresol}, there is
	\begin{equation*}
		\begin{aligned}
			\mathbb{E}\|u-u_{N}\|_{\mathbb{H}}^{2}\leq &C\mathbb{E}\left \|\int_{0}^{t}\mathcal{R}(t-r)f(u(r))-\mathcal{R}_{N}(t-r)P_{N}f(u_{N}(r))dr\right \|_{\mathbb{H}}^{2}\\
			&+C\mathbb{E}\left \|\sum_{k=N+1}^{\infty}\int_{0}^{t}\sqrt{\Lambda_{k}}E_{k}(t-r)\phi_{k}dW^{H}_{k}(r)\right \|_{\mathbb{H}}^{2}\\
			\leq& \uppercase\expandafter{\romannumeral1}+\uppercase\expandafter{\romannumeral2}.
		\end{aligned}
	\end{equation*}

	For $\uppercase\expandafter{\romannumeral1}$, we split it into two parts
	\begin{equation*}
		\begin{aligned}
			\uppercase\expandafter{\romannumeral1}\leq &C\mathbb{E}\left \|\int_{0}^{t}\mathcal{R}(t-r)(f(u(r))-f(u_{N}(r)))dr\right \|_{\mathbb{H}}^{2}\\
			&+C\mathbb{E}\left \|\int_{0}^{t}(\mathcal{R}(t-r)-\mathcal{R}_{N}(t-r)P_{N})f(u_{N}(r))dr\right \|_{\mathbb{H}}^{2}\\
			\leq& \uppercase\expandafter{\romannumeral1}_{1}+\uppercase\expandafter{\romannumeral1}_{2}.
		\end{aligned}
	\end{equation*}
	The resolvent estimate \eqref{equresolvent0}  gives
	\begin{equation*}
		\begin{aligned}
			\uppercase\expandafter{\romannumeral1}_{1}\leq& C\mathbb{E}\left( \int_{0}^{t}\|f(u(r))-f(u_{N}(r))\|_{\mathbb{H}}dr \right)^{2}\\
			\leq& C \int_{0}^{t}\mathbb{E}\|u(r)-u_{N}(r)\|_{\mathbb{H}}^{2}dr.
		\end{aligned}
	\end{equation*}
	By the definitions of $\mathcal{R}_{N}$, $P_{N}$, and Theorem \ref{thmsoboun}, one has
	\begin{equation*}
		\begin{aligned}
			\uppercase\expandafter{\romannumeral1}_{2}\leq& C\mathbb{E}\sum_{k=N+1}^{\infty}\left (\int_{0}^{t}E_{k}(t-r)(f(u_{N}(r)),\phi_{k})dr\right )^{2}\\
			\leq& C\mathbb{E}\sum_{k=N+1}^{\infty}\lambda_{k}^{-2\sigma}\left (\int_{0}^{t}E_{k}(t-r)\lambda_{k}^{\sigma}(f(u_{N}(r)),\phi_{k})dr\right )^{2}\\
			\leq&C\sup_{k\geq N+1} \lambda_{k}^{-2\sigma}\int_{0}^{t}(t-r)^{1-\epsilon}\left (\lambda_{k}^{\sigma}E_{k}(t-r)\right )^{2}dr\\
			\leq&C\sup_{k\geq N+1} \lambda_{k}^{-2\sigma},
		\end{aligned}
	\end{equation*}
	where $\sigma< \frac{s}{\alpha}$. As for $\uppercase\expandafter{\romannumeral2}$, one can get
	\begin{equation*}
		\begin{aligned}
			\uppercase\expandafter{\romannumeral2}
			\leq& C\sum_{k=N+1}^{\infty}\int_{0}^{t}\left |\sqrt{\Lambda_{k}}{}_{0}\partial^{\frac{1-2H}{2}}_{r}E_{k}(r)\right |^{2}dr\\
			\leq& C\sup_{k\geq N+1}\int_{0}^{t}\left |\lambda_{k}^{\rho}{}_{0}\partial^{\frac{1-2H}{2}}_{r}E_{k}(r)\right |^{2}dr,
		\end{aligned}
	\end{equation*}
	where we use Lemma \ref{eqcorleml} and $\|A^{-\rho}Q^{1/2}\|_{\mathcal{L}_{2}}<\infty$. Then Eq. \eqref{equresolvent} gives
	\begin{equation*}
		\begin{aligned}
			&\int_{0}^{t}\left |\lambda_{k}^{\rho}{}_{0}\partial^{\frac{1-2H}{2}}_{r}E_{k}(r)\right |^{2}dr\\
			\leq&C\int_{0}^{t}\left |\int_{\Gamma_{\theta, \kappa}}e^{zr}\lambda_{k}^{\rho}z^{\frac{1-2H}{2}}\tilde{E}_{k}(z)dz\right |^{2}dr\\
			\leq&C\lambda_{k}^{-2\sigma}\int_{0}^{t}\left (\int_{\Gamma_{\theta, \kappa}}|e^{zr}||z|^{\frac{1-2H}{2}}|\lambda_{k}^{\rho+\sigma}\tilde{E}_{k}(z)||dz|\right )^{2}dr\\
			\leq &C\lambda_{k}^{-2\sigma}\int_{0}^{t}\left (\int_{\Gamma_{\theta,\kappa}}|e^{zr}||z|^{\frac{1-2H}{2}+(\sigma+\rho)\alpha/s-1}|dz|\right )^{2}dr\\
			\leq& C\lambda_{k}^{-2\sigma}\int_{0}^{t} r^{2H-1-2(\sigma+\rho)\alpha/s}dr,
		\end{aligned}
	\end{equation*}
	where we need to require $2H-1-2(\sigma+\rho)\alpha/s>-1$, i.e., $\sigma<\frac{H s}{\alpha}-\rho$. Thus, by the Gr\"onwall inequality and Lemma \ref{thmeigenvalue}, the desired result is obtained.
\end{proof}

\section{Time discretization and error analysis}
Here backward Euler convolution quadrature \cite{Lubich.1988CqadocIb,Lubich.1988CqadocI,Lubich.1996Ndeefaoaeewaptmt} is used to discretize the Riemann-Liouville fractional derivative and the corresponding error estimate is also provided.
Let the time step size $\tau=T/L$ with $L\in\mathbb{N}^{*}$, $t_i=i\tau$, $i=0,1,\ldots,L$, and $0=t_0<t_1<\cdots<t_L=T$. Introduce  the operator $\bar{\partial}_{\tau}$ as
\begin{equation}
	\bar{\partial}_{\tau} u(t)=\left\{
	\begin{aligned}
		&0\qquad t=t_{0},\\
		&\frac{u(t_{j})-u(t_{j-1})}{\tau}\qquad t\in(t_{j-1},t_{j}].
	\end{aligned} \right.
\end{equation}

Using backward Euler method to discretize the corresponding temporal operator, the fully-discrete scheme of \eqref{equretosol} can be written as
\begin{equation}\label{eqfullscheme}
	\frac{u^{n}_{N}-u^{n-1}_{N}}{\tau}+\sum_{i=0}^{n-1}d^{(1-\alpha)}_{i}A^{s}_{N}u^{n-i}_{N}=P_{N} f(u_{N}^{n-1})+P_{N}\bar{\partial}_{\tau} W^{H}_{Q}(t_{n}),
\end{equation}
where
\begin{equation*}
	(\delta_{\tau}(\zeta))^{\alpha}=\sum_{i=0}^{\infty}d^{(\alpha)}_{i}\zeta^{i},\quad \delta_{\tau}(\zeta)=\frac{1-\zeta}{\tau}.
\end{equation*}
Denote $\bar{F}(t)$ as
\begin{equation*}
	\bar{F}(t)=\left\{
	\begin{aligned}
		&0\qquad\qquad~~~ t=t_{0},\\
		&f(u_{N}^{j-1})\qquad t\in(t_{j-1},t_{j}],\\
	\end{aligned} \right.
\end{equation*}
and $F(t)=f(u_{N}(t))$. In the following, we abbreviate $P_{N}F(t)$ and $P_{N}\bar{F}(t)$ as $F_{N}$ and $\bar{F}_{N}$.
With the help of the facts \cite{Gunzburger.2018ScrotdfstfPstastwn}
\begin{equation*}
	\sum_{n=1}^{\infty}\bar{\partial}_{\tau} W^{H}_{Q}(t_{n})e^{-zt_{n}}=\frac{z}{e^{z\tau}-1} \widetilde{\bar{\partial}_{\tau}W^{H}_{Q}},\quad \sum_{n=1}^{\infty}\bar{F}_{N}(t_{n})e^{-zt_{n}}=\frac{z}{e^{z\tau}-1}\tilde{\bar{F}}_{N}(z)
\end{equation*}
and doing simple calculations, we can get
\begin{equation}\label{eqfullschsolrep}
	u^{n}_{N}=\int_{0}^{t_{n}}\bar{\mathcal{R}}_{N}(t_{n}-r)\bar{F}_{N}(r)dr+\sum_{k=1}^{N}\int_{0}^{t_{n}}\sqrt{\Lambda_{k}}\bar{E}_{k}(t_{n}-r)\phi_{k} \bar{\partial}_{\tau}W^{H}_{k}(r)dr,
\end{equation}
where
\begin{equation*}
	\bar{\mathcal{R}}_{N}(t)=\frac{1}{2\pi\mathbf{i}}\int_{\Gamma^{\tau}_{\theta,\kappa}}e^{zt}(\delta_{\tau}(e^{-z\tau}))^{\alpha-1}((\delta_{\tau}(e^{-z\tau}))^{\alpha}+A_{N}^{s})^{-1}\frac{z\tau}{e^{z\tau}-1} dz
\end{equation*}
and
\begin{equation*}
	\bar{E}_{k}(t)=\frac{1}{2\pi\mathbf{i}}\int_{\Gamma^{\tau}_{\theta,\kappa}}e^{zt}(\delta_{\tau}(e^{-z\tau}))^{\alpha-1}((\delta_{\tau}(e^{-z\tau}))^{\alpha}+\lambda_{k}^{s})^{-1}\frac{z\tau}{e^{z\tau}-1} dz.
\end{equation*}
Here $\Gamma^{\tau}_{\theta,\kappa}=\{z\in \mathbb{C}:\kappa\leq|z|\leq  \frac{\pi}{\tau \sin(\theta)},|\arg z|=\theta\}\cup\{z\in \mathbb{C}:|z|=\kappa,|\arg z|\leq \theta\}$.

To estimate $\mathbb{E}\|u_{N}(t_{n})-u^{n}_{N}\|_{\mathbb{H}}^{2}$, the following lemma is needed.
\begin{lemma}[\cite{Gunzburger.2018ScrotdfstfPstastwn}]\label{Lemseriesest}
	Let the given $\alpha\in(0,1)$ and $\theta \in\left(\frac{\pi}{2}, \operatorname{arccot}\left(-\frac{2}{\pi}\right)\right)$,  where $arccot$ means the inverse function of $\cot$, and a fixed $\xi \in (0,1)$. Then, when $z$ lies in the region enclosed by $\Gamma^\tau_\xi=\{z=-\ln{(\xi)}/\tau+\mathbf{i}y:y\in\mathbb{R}~and~|y|\leq \pi/\tau\}$, $\Gamma^\tau_{\theta,\kappa}$, and the two lines $\mathbb{R}\pm \mathbf{i}\pi/\tau$, whenever $0<\kappa \leq \min (1 / T,-\ln (\xi) / \tau)$, $\delta_\tau(e^{-z\tau})$ and $(\delta_\tau(e^{-z\tau})+A)^{-1}$ are both analytic. Furthermore, we have
	\begin{equation*}
		\begin{aligned}
			&\delta_{\tau}\left(e^{-z \tau}\right) \in \Sigma_{\theta}&\forall z \in \Gamma_{\theta, \kappa}^{\tau},\\
			&C_{0}|z| \leq\left|\delta_{\tau}\left(e^{-z\tau }\right)\right| \leq C_{1}|z|&\forall z \in \Gamma_{\theta, \kappa}^{\tau},\\
			&\left|\delta_{\tau}\left(e^{-z\tau }\right)-z\right| \leq C \tau|z|^{2}&\forall z \in \Gamma_{\theta, \kappa}^{\tau},\\
			&\left|\delta_{\tau}\left(e^{-z\tau }\right)^{\alpha}-z^{\alpha}\right| \leq C \tau|z|^{\alpha+1}&\forall z \in \Gamma_{\theta, \kappa}^{\tau},
		\end{aligned}
	\end{equation*}
	where $\kappa\in (0,\min (1 / T,-\ln (\xi) / \tau))$ and the constants $C_0$, $C_1$, and $C$ are independent of $\tau$.
	
\end{lemma}

Introduce $G_{k}$ as
\begin{equation}\label{eqdefFk}
	G_{k}(t_{i-1})=\frac{1}{\tau}\int_{t_{i-1}}^{t_{i}}\bar{E}_{k}(r)dr\qquad t\in[t_{i-1},t_{i}),
\end{equation}
and denote $G_{k}(t_{i})$ as $G_{k,i}$.
According to the definition of $\bar{E}_{k}(r)$, we have
\begin{equation*}
	\begin{aligned}
		G_{k,i-1}=&\frac{1}{\tau}\int_{t_{i-1}}^{t_{i}}\frac{1}{2\pi \mathbf{i}}\int_{\Gamma_{\theta, \kappa}^{\tau}}e^{zr}\tilde{\bar{E}}_{k}(z)dzdr\\
		=&\frac{1}{2\pi \mathbf{i}}\int_{\Gamma_{\theta, \kappa}^{\tau}}\frac{e^{zt_{i}}-e^{zt_{i-1}}}{\tau z}\tilde{\bar{E}}_{k}(z)dz\\
		=&\frac{1}{2\pi\mathbf{i}}\int_{\Gamma^{\tau}_{\theta,\kappa}}e^{zt_{i-1}}(\delta_{\tau}(e^{-z\tau}))^{\alpha-1}((\delta_{\tau}(e^{-z\tau}))^{\alpha}+\lambda_{k}^{s})^{-1} dz.
	\end{aligned}
\end{equation*}
Simple calculations lead to
\begin{equation*}
	\sum_{i=0}^{\infty}G_{k,i}\zeta^{i}=\frac{1}{\tau}(\delta_{\tau}(\zeta))^{\alpha-1}((\delta_{\tau}(\zeta))^{\alpha}+\lambda_{k}^{s})^{-1}.
\end{equation*}
Thus the Laplace transform of $G_{k}$ can be written as
\begin{equation}
	\begin{aligned}
		\tilde{G}_{k}(z)=&\int_{0}^{\infty}e^{-zt}G_{k}(t)dt=\sum_{i=0}^{\infty}G_{k,i}\int_{t_{i}}^{t_{i+1}}e^{-zt}dt\\
		&\quad=\sum_{i=0}^{\infty}G_{k,i}e^{-zt_{i}}\frac{1-e^{-z\tau}}{z}
		=\frac{1}{z}(\delta_{\tau}(e^{-z\tau}))^{\alpha}((\delta_{\tau}(e^{-z\tau}))^{\alpha}+\lambda_{k}^{s})^{-1}.
	\end{aligned}
\end{equation}
Then we provide an estimate of $\tilde{G}_{k}(z)$.
\begin{lemma}\label{LemEsF}
	Let $G_{k}$ be defined in \eqref{eqdefFk}. For $\tau<\tau^{*}$ $($the value of $\tau^{*}$ depends on $\lambda_{k}$$)$, one has
	\begin{equation*}
		|\lambda_{k}^{\beta s}\tilde{G}_{k}(z)|\leq \left\{
		\begin{aligned}
			&	C|z|^{\beta\alpha-1}e^{\beta\alpha|z|\tau}\quad z\in\Gamma_{\theta, \kappa}\backslash\Gamma_{\theta,\kappa}^{\tau},\\
			&	C|z|^{\beta\alpha-1}\qquad\qquad z\in\Gamma_{\theta,\kappa}^{\tau},
		\end{aligned}		
		\right.
	\end{equation*}
	where $\beta\in[0,1]$, $s\in(0,1)$, and $\alpha\in(0,1)$.
\end{lemma}
\begin{proof}
	When $z\in\Gamma_{\theta, \kappa}^{\tau}$, the desired estimate can be got by Eq. \eqref{equresolvent} and Lemma \ref{Lemseriesest}.
	
	As for $z=|z| e^{\theta\mathbf{i}}\in\Gamma_{\theta,\kappa}\backslash\Gamma_{\theta,\kappa}^{\tau}$, we consider $\tau\delta_{\tau}(e^{-z\tau})$ first. Simple calculations give
	\begin{equation*}
		\begin{aligned}
			|\tau\delta_{\tau}(e^{-z\tau})|=&|1-e^{-z\tau}|\\
			=&\left |1-e^{-|z|\cos(\theta)\tau}e^{-\mathbf{i}|z|\sin(\theta)\tau}\right |.
		\end{aligned}
	\end{equation*}
	Since $z\in \Gamma_{\theta,\kappa}\backslash\Gamma_{\theta,\kappa}^{\tau}$, one has $|z|\geq \frac{\pi}{\tau\sin(\theta)}$. Choosing a suitable $\theta\in(\frac{\pi}{2},\pi)$ satisfying $\cot(\theta)\pi\leq-1$, we have
	\begin{equation*}
		|\tau\delta_{\tau}(e^{-z\tau})|\geq e-1,
	\end{equation*}
	which yields
	\begin{equation*}
		|\delta_{\tau}(e^{-z\tau})|\geq \frac{e-1}{\tau}.
	\end{equation*}
	Let $\tau$ be small enough to satisfy $(\frac{e-1}{\tau})^{\alpha}>2\lambda_{k}^{s}$. Then it has
	\begin{equation*}
		\begin{aligned}
			\left |\lambda_{k}^{\beta s}\frac{1}{(\delta_{\tau}(e^{-z\tau}))^{\alpha}+\lambda^{s}_{k}}\right |\leq C|\delta_{\tau}(e^{-z\tau})|^{(\beta-1)\alpha}.
		\end{aligned}
	\end{equation*}
	Combining the definition of $G_{k}$, one can get
	\begin{equation*}
		|\tilde{G}_{k}|\leq C|z|^{-1}|\delta_{\tau}(e^{-z\tau})|^{\beta\alpha}.
	\end{equation*}
	According to the fact $\delta_{\tau}(e^{-z\tau})\leq |z|\sum_{k=1}^{\infty}\frac{|z\tau|^{k-1}}{k!}\leq |z|e^{|z|\tau}$, $\forall z\in\Gamma_{\theta,\kappa}$ \cite{Jin.2020Iisos}, the desired result is reached.
\end{proof}

In the rest of paper, we take $\kappa\leq\frac{\pi}{t_{n}|\sin(\theta)|}$.
Then we provide the temporal error estimate.

\begin{theorem}\label{thmtimeer}
	Let $u_{N}(t_{n})$ and $u_{N}^{n}$ be the solutions of Eqs. \eqref{eqsemischeme} and \eqref{eqfullscheme}, respectively, and $\|A^{-\rho}Q^{1/2}\|_{\mathcal{L}_{2}}<\infty$ with $\rho\in \left[0,\frac{sH}{\alpha}\right)\cap[0,s]$ and $\alpha\in(0,1)$. Then there holds
	\begin{equation*}
		\left (\mathbb{E}\|u_{N}(t_{n})-u_{N}^{n}\|_{\mathbb{H}}^{2}\right )^{1/2}\leq C\tau^{H-\frac{\rho\alpha}{s}-\epsilon}.
	\end{equation*}
\end{theorem}
\begin{proof}
	Using \eqref{eqsemiresol} and \eqref{eqfullschsolrep} and taking the expectation of $\|u_{N}(t_{n})-u_{N}^{n}\|^{2}_{\mathbb{H}}$ yield
	\begin{equation*}
		\begin{aligned}
			&\mathbb{E}\|u_{N}(t_{n})-u_{N}^{n}\|_{\mathbb{H}}^{2}\\
			=&\mathbb{E}\left \|\int_{0}^{t_{n}}\mathcal{R}_{N}(t_{n}-r)F_{N}(r)-\bar{\mathcal{R}}_{N}(t_{n}-r)\bar{F}_{N}(r)dr\right \|_{\mathbb{H}}^{2}\\
			&\!\!\!\!\!\!+\mathbb{E}\left \|\sum_{k=1}^{N}\left (\int_{0}^{t_{n}}\sqrt{\Lambda_{k}}E_{k}(t_{n}-r)\phi_{k}dW^{H}_{k}(r)\right .\right .\\
			&\qquad\qquad\left .\left .-\int_{0}^{t_{n}}\sqrt{\Lambda_{k}}\bar{E}_{k}(t_{n}-r)\phi_{k}\bar{\partial}_{\tau}W^{H}_{k}(r)dr\right )\right \|_{\mathbb{H}}^{2}\\
			\leq&\mathbb{E}\left \|\int_{0}^{t_{n}}\mathcal{R}_{N}(t_{n}-r)F_{N}(r)-\bar{\mathcal{R}}_{N}(t_{n}-r)\bar{F}_{N}(r)dr\right \|_{\mathbb{H}}^{2}\\
			&+\mathbb{E}\left \|\int_{0}^{t_{n}}\sum_{k=1}^{N}\sqrt{\Lambda_{k}}(E_{k}(t_{n}-r)-\bar{E}_{k}(t_{n}-r))\phi_{k}dW^{H}_{k}(r)\right \|_{\mathbb{H}}^{2}\\
			&+\mathbb{E}\left \|\int_{0}^{t_{n}}\sum_{k=1}^{N}\sqrt{\Lambda_{k}}\bar{E}_{k}(t_{n}-r)\phi_{k}\left(dW^{H}_{k}(r)-\bar{\partial}_{\tau}W^{H}_{k}(r)dr \right)\right \|_{\mathbb{H}}^{2}\\
			=&\vartheta_{1}+\vartheta_{2}+\vartheta_{3}.
		\end{aligned}
	\end{equation*}
	For $\vartheta_{1}$, one can split it as
	\begin{equation*}
		\begin{aligned}
			\vartheta_{1}\leq& \mathbb{E}\left \|\sum_{i=1}^{n}\int_{t_{i-1}}^{t_{i}}\mathcal{R}_{N}(t_{n}-r)(F_{N}(r)-F_{N}(t_{i-1}))dr\right \|_{\mathbb{H}}^{2}\\
			&+ \mathbb{E}\left \|\sum_{i=1}^{n}\int_{t_{i-1}}^{t_{i}}(\mathcal{R}_{N}(t_{n}-r)-\bar{\mathcal{R}}_{N}(t_{n}-r))F_{N}(t_{i-1})dr\right \|_{\mathbb{H}}^{2}\\
			&+ \mathbb{E}\left \|\sum_{i=1}^{n}\int_{t_{i-1}}^{t_{i}}\bar{\mathcal{R}}_{N}(t_{n}-r)(F_{N}(t_{i-1})-\bar{F}_{N}(t_{i}))dr\right \|_{\mathbb{H}}^{2}\\
			=&\vartheta_{1,1}+\vartheta_{1,2}+\vartheta_{1,3}.
		\end{aligned}
	\end{equation*}
	By the  assumption \eqref{eqassptf} and Theorem \ref{thmholderun}, there holds
	\begin{equation*}
		\vartheta_{1,1}\leq C\tau^{2H-\frac{2\rho\alpha}{s}-\epsilon}.
	\end{equation*}
	As for $\vartheta_{1,2}$, one can get
	\begin{equation*}
		\begin{aligned}
			\vartheta_{1,2}\leq& C\mathbb{E}\sum_{i=1}^{n}\int_{t_{i-1}}^{t_{i}}(t_{n}-r)^{1-\epsilon}\|\mathcal{R}_{N}(t_{n}-r)-\bar{\mathcal{R}}_{N}(t_{n}-r)\|^{2}\left \|F_{N}(t_{i})\right \|_{\mathbb{H}}^{2}dr\\
			\leq&C\int_{0}^{t_{n}}(t_{n}-r)^{1-\epsilon}\|\mathcal{R}_{N}(t_{n}-r)-\bar{\mathcal{R}}_{N}(t_{n}-r)\|^{2}dr\\
			\leq&C\int_{0}^{t_{n}}r^{1-\epsilon}\left (\int_{\Gamma_{\theta,\kappa}\backslash\Gamma_{\theta, \kappa}^{\tau}}|e^{zr}|\|\tilde{\mathcal{R}}_{N}(z)\||dz|\right )^{2}dr\\
			&+C\int_{0}^{t_{n}}r^{1-\epsilon}\left (\int_{\Gamma_{\theta, \kappa}^{\tau}}|e^{zr}|\|\tilde{\mathcal{R}}_{N}(z)-\tilde{\bar{\mathcal{R}}}_{N}(z)\||dz|\right )^{2}dr\\
			= &\vartheta_{1,2,1}+\vartheta_{1,2,2}.
		\end{aligned}
	\end{equation*}
	As for $\vartheta_{1,2,1}$, there holds
	\begin{equation*}
		\begin{aligned}
			\vartheta_{1,2,1}\leq& C\int_{0}^{t_{n}}r^{1-\epsilon}\left (\int_{\Gamma_{\theta,\kappa}\backslash\Gamma_{\theta, \kappa}^{\tau}}|e^{zr}||z|^{-1}|dz|\right )^{2}dr\\
			\leq& C\tau^{2-2\epsilon}\int_{0}^{t_{n}}r^{1-\epsilon}\left (\int_{\Gamma_{\theta,\kappa}\backslash\Gamma_{\theta, \kappa}^{\tau}}|e^{zr}||z|^{-\epsilon}|dz|\right )^{2}dr\\
			\leq& C\tau^{2-2\epsilon}.
		\end{aligned}
	\end{equation*}
	Similarly, one has
	\begin{equation*}
		\begin{aligned}
			\vartheta_{1,2,2}\leq& C\tau^{2}\int_{0}^{t_{n}}r^{1-\epsilon}\left (\int_{\Gamma_{\theta, \kappa}^{\tau}}|e^{zr}||dz|\right )^{2}dr\\
			\leq& C\tau^{2}\int_{0}^{t_{n}}r^{1-\epsilon}\int_{\Gamma_{\theta, \kappa}^{\tau}}|e^{2zr}||z|^{1-2\epsilon}|dz|\int_{\Gamma_{\theta, \kappa}^{\tau}}|z|^{-1+2\epsilon}|dz|dr\\
			\leq& C\tau^{2-2\epsilon}.
		\end{aligned}
	\end{equation*}
	As for $\vartheta_{1,3}$, one can get
	\begin{equation*}
		\vartheta_{1,3}\leq C\tau\sum_{i=1}^{n-1}\mathbb\|u_{N}(t_{i})-u_{N}^{i}\|_{\mathbb{H}}^{2}.
	\end{equation*}
	For $\vartheta_{2}$, we have
	\begin{equation*}
		\begin{aligned}
			\vartheta_{2}\leq& C\sum_{k=1}^{N}\int_{0}^{t_{n}}(\sqrt{\Lambda_{k}}{}_{0}\partial_{r}^{\frac{1-2H}{2}}(E_{k}(r)-\bar{E}_{k}(r)))^{2}dr\\
			\leq& C\sup_{1\leq k\leq N}\int_{0}^{t_{n}}\left (\lambda_{k}^{\rho}{}_{0}\partial_{r}^{\frac{1-2H}{2}}(E_{k}(r)-\bar{E}_{k}(r))\right )^{2}dr.
		\end{aligned}
	\end{equation*}
	Then we split the following formula into two parts
	\begin{equation*}
		\begin{aligned}
			&\int_{0}^{t_{n}}\left (\lambda_{k}^{\rho}{}_{0}\partial_{r}^{\frac{1-2H}{2}}(E_{k}(r)-\bar{E}_{k}(r))\right )^{2}dr\\
			\leq &\int_{0}^{t_{n}}\left |\int_{\Gamma_{\theta,\kappa}\backslash\Gamma_{\theta, \kappa}^{\tau}}e^{zr}\lambda_{k}^{\rho}z^{\frac{1-2H}{2}}\tilde{E}_{k}(z)dz\right |^{2}dr\\
			&+\int_{0}^{t_{n}}\left |\int_{\Gamma_{\theta, \kappa}^{\tau}}e^{zr}\lambda_{k}^{\rho}z^{\frac{1-2H}{2}}(\tilde{E}_{k}(z)-\tilde{\bar{E}}_{k}(z))dz\right |^{2}dr\\
			=& \vartheta_{2,1}+\vartheta_{2,2}.
		\end{aligned}
	\end{equation*}
	From \eqref{equresolvent}, it follows that
	\begin{equation*}
		\begin{aligned}
			\vartheta_{2,1}\leq& C\int_{0}^{t_{n}}\left (\int_{\Gamma_{\theta,\kappa}\backslash\Gamma_{\theta, \kappa}^{\tau}}|e^{zr}||z|^{\frac{1-2H}{2}+\frac{\rho\alpha}{s}-1}|dz|\right )^{2}dr\\
			\leq& C\tau^{2H-\frac{2\rho\alpha}{s}-\epsilon}\int_{0}^{t_{n}}\int_{\Gamma_{\theta,\kappa}\backslash\Gamma_{\theta, \kappa}^{\tau}}|e^{2zr}||z|^{-\epsilon}|dz|dr\\
			\leq& C\tau^{2H-\frac{2\rho\alpha}{s}}.
		\end{aligned}
	\end{equation*}
	Combining Lemma \ref{Lemseriesest} and \eqref{equresolvent} leads to
	\begin{equation*}
		\begin{aligned}
			\vartheta_{2,2}\leq &C\tau^{2}\int_{0}^{t_{n}}\left (\int_{\Gamma_{\theta, \kappa}^{\tau}}|e^{zr}||z|^{\frac{1-2H}{2}+\frac{\rho\alpha}{s}}dz\right )^{2}dr\\
			\leq &C\tau^{2}\int_{0}^{t_{n}}\int_{\Gamma_{\theta, \kappa}^{\tau}}|e^{2zr}||z|^{1-2H+\frac{2\rho\alpha}{s}}|dz| dr\int_{\Gamma_{\theta, \kappa}^{\tau}}|dz|\\
			\leq& C\tau^{2H-\frac{2\rho\alpha}{s}}.
		\end{aligned}
	\end{equation*}
	Lemma \ref{eqcorleml} gives
	\begin{equation*}
		\begin{aligned}
			&\vartheta_{3}\\\leq& C\mathbb{E}\left \|\int_{0}^{t_{n}}\sum_{k=1}^{N}\sqrt{\Lambda_{k}}\Bigg (\bar{E}_{k}(t_{n}-r)\right .\\
			&\qquad\qquad\left .-\frac{1}{\tau}\sum_{i=1}^{n}\chi_{(t_{i-1},t_{i}]}(r)\int_{t_{i-1}}^{t_{i}}\bar{E}_{k}(t_{n}-\xi)d\xi\Bigg )\phi_{k}dW^{H}_{k}(r)\right \|^{2}_{\mathbb{H}}\\
			\leq&C\sum_{k=1}^{N}\int_{0}^{t_{n}}\left|\sqrt{\Lambda_{k}}{}_{0}\partial^{\frac{1-2H}{2}}_{r}\left (\bar{E}_{k}(r)-G_{k}(r)\right )\right|^{2}dr\\
			\leq&C\sup_{1\leq k\leq N}\int_{0}^{t_{n}}\left|\lambda_{k}^{\rho}{}_{0}\partial^{\frac{1-2H}{2}}_{r}\left (\bar{E}_{k}(r)-G_{k}(r)\right )\right|^{2}dr.
		\end{aligned}
	\end{equation*}
	Divide $\vartheta_{3}$ into three parts
	\begin{equation*}
		\begin{aligned}
			&\int_{0}^{t_{n}}\left|\lambda_{k}^{\rho}{}_{0}\partial^{\frac{1-2H}{2}}_{r}\left (\bar{E}_{k}(r)-G_{k}(r)\right )\right|^{2}dr\\
			\leq&C\int_{t_{1}}^{t_{n}}\left|\int_{\Gamma_{\theta,\kappa}^{\tau}}e^{zr}\lambda_{k}^{\rho}z^{\frac{1-2H}{2}}\tilde{\bar{E}}_{k}(z)dz-\int_{\Gamma_{\theta,\kappa}}e^{zr}\lambda_{k}^{\rho}z^{\frac{1-2H}{2}}\tilde{G}_{k}(z)dz\right|^{2}dr\\
			&+\int_{0}^{t_{1}}\left|\lambda_{k}^{\rho}{}_{0}\partial^{\frac{1-2H}{2}}_{r}\left (\bar{E}_{k}(r)-G_{k}(r)\right )\right|^{2}dr\\
			\leq&C\int_{t_{1}}^{t_{n}}\left|\int_{\Gamma_{\theta,\kappa}^{\tau}}e^{zr}\lambda_{k}^{\rho}z^{\frac{1-2H}{2}}(\tilde{\bar{E}}_{k}(z)-\tilde{G}_{k}(z))dz\right |^{2}dr\\
			&+C\int_{t_{1}}^{t_{n}}\left|\int_{\Gamma_{\theta,\kappa}\backslash\Gamma_{\theta,\kappa}^{\tau}}e^{zr}\lambda_{k}^{\rho}z^{\frac{1-2H}{2}}\tilde{G}_{k}(z)dz\right|^{2}dr\\
			&+\int_{0}^{t_{1}}\left|\lambda_{k}^{\rho}{}_{0}\partial^{\frac{1-2H}{2}}_{r}\left (\bar{E}_{k}(r)-G_{k}(r)\right )\right|^{2}dr\\			
			=& \vartheta_{3,1}+\vartheta_{3,2}+\vartheta_{3,3}.
		\end{aligned}
	\end{equation*}
	Eq. \eqref{equresolvent} gives
	\begin{equation*}
		\begin{aligned}
			\vartheta_{3,1}\leq&C\tau^{2}\int_{t_{1}}^{t_{n}}\left|\int_{\Gamma_{\theta,\kappa}\backslash\Gamma_{\theta,\kappa}^{\tau}}e^{zr}z^{\frac{1-2H}{2}+\frac{\rho\alpha}{s}}dz\right|^{2}dr\\
			\leq &C\tau^{2}\int_{t_{1}}^{t_{n}}\int_{\Gamma_{\theta, \kappa}^{\tau}}|e^{2zr}||z|^{1-2H+\frac{2\rho\alpha}{s}}|dz| dr\int_{\Gamma_{\theta, \kappa}^{\tau}}|dz|\\
			\leq& C\tau^{2H-\frac{2\rho\alpha}{s}}.
		\end{aligned}
	\end{equation*}
	As for $\vartheta_{3,2}$, using Lemma \ref{LemEsF}, one has
	\begin{equation*}
		\begin{aligned}
			\vartheta_{3,2}\leq& C\int_{t_{1}}^{t_{n}}\left (\int_{\Gamma_{\theta,\kappa}\backslash\Gamma_{\theta, \kappa}^{\tau}}|e^{zr+|z|\frac{\rho\alpha\tau}{s}}||z|^{\frac{1-2H}{2}+\frac{\rho\alpha}{s}-1}|dz|\right )^{2}dr\\
			\leq& C\tau^{2H-\frac{2\rho\alpha}{s}-\epsilon}\int_{t_{1}}^{t_{n}}\int_{\Gamma_{\theta,\kappa}\backslash\Gamma_{\theta, \kappa}^{\tau}}|e^{2zr+2|z|\frac{\rho\alpha\tau}{s}}||z|^{-\epsilon}|dz|dr\\
			\leq& C\tau^{2H-\frac{2\rho\alpha}{s}}.
		\end{aligned}
	\end{equation*}
 	As for $\vartheta_{3,3}$, we consider $\lambda^{\rho}_{k}(\bar{E}_{k}(r_{1})-\bar{E}_{k}(r_{2}))$ with $|r_{1}-r_{2}|\leq \tau$ and $r_{1},r_{2}>0$ first. Simple calculations give
 	\begin{equation*}
 		\begin{aligned}
 			&|\lambda^{\rho}_{k}(\bar{E}_{k}(r_{1})-\bar{E}_{k}(r_{2}))|\\
 			\leq&C\left|\int_{\Gamma^{\tau}_{\theta,\kappa}}(e^{zr_{1}}-e^{zr_{2}})\lambda^{\rho}_{k} \tilde{\bar{E}}_{k}(z)dz\right |\\
 			\leq& C\tau^{\gamma}\int_{\Gamma^{\tau}_{\theta,\kappa}}|e^{z(r_{1}+\tau)}||z|^{\rho\alpha/s-1+\gamma}|dz|\\
 			\leq& C\tau^{\gamma}(r_{1}+\tau)^{-\gamma-\rho\alpha/s},
 		\end{aligned}
 	\end{equation*}
 where $\gamma\in[0,1-\rho\alpha/s)$.
 Thus when $H=1/2$, the mean value theorem gives
 \begin{equation*}
 	\begin{aligned}
 		\vartheta_{3,3}\leq C\tau^{1-2\rho\alpha/s}.
 	\end{aligned}
 \end{equation*}
 	 Similarly, when $H\in(1/2,1)$, we have
 	\begin{equation*}
 		\begin{aligned}
 			\vartheta_{3,3}\leq& \int_{0}^{t_{1}}\left(\lambda_{k}^{\rho}\int_{0}^{r}(r-r_{1})^{\frac{2H-1}{2}-1}\left |\bar{E}_{k}(r_{1})-\frac{1}{\tau}\int_{0}^{t_{1}}\bar{E}_{k}(\xi)d\xi\right |dr_{1}\right)^{2}dr\\
 			\leq& C\tau^{2H-2\rho\alpha/s}.
 		\end{aligned}
 	\end{equation*}
 As for $H\in(0,1/2)$, one has
 \begin{equation*}
 	\begin{aligned}
 		\vartheta_{3,3}\leq& \int_{0}^{t_{1}}\left(\lambda_{k}^{\rho}\partial_{t}\int_{0}^{r}(r-r_{1})^{\frac{2H-1}{2}}\left |\bar{E}_{k}(r_{1})-\frac{1}{\tau}\int_{0}^{t_{1}}\bar{E}_{k}(\xi)d\xi\right |dr_{1}\right)^{2}dr\\
 		\leq& \int_{0}^{t_{1}}\left|\lambda_{k}^{\rho}\int_{0}^{r}(r-r_{1})^{\frac{2H-1}{2}} \partial_{r_{1}}\left(\bar{E}_{k}(r_{1})-\frac{1}{\tau}\int_{t_{2}}^{t_{3}}\bar{E}_{k}(\xi)d\xi\right)dr_{1}\right|^{2}dr\\
 		&+\int_{0}^{t_{1}}\left|\lambda_{k}^{\rho}\partial_{t}\int_{0}^{r}(r-r_{1})^{\frac{2H-1}{2}}\left (\bar{E}_{k}(0)-\frac{1}{\tau}\int_{0}^{t_{1}}\bar{E}_{k}(\xi)d\xi\right )dr_{1}\right|^{2}dr\\
 		\leq& C\tau^{2H-2\rho\alpha/s},
 	\end{aligned}
 \end{equation*}
where we use that for $\gamma\in[0,1]$,
\begin{equation*}
	\begin{aligned}
		&|\lambda^{\rho}_{k}\partial_{r_{1}}(\bar{E}_{k}(r_{1})-\bar{E}_{k}(r_{1}+\tau))|\\
		\leq&C\left|\partial_{r_{1}}\int_{\Gamma^{\tau}_{\theta,\kappa}}e^{zr_{1}}(1-e^{z\tau})\lambda^{\rho}_{k} \tilde{\bar{E}}_{k}(z)dz\right |\\
		\leq& C\tau^{\gamma}\int_{\Gamma^{\tau}_{\theta,\kappa}}|e^{z(r_{1}+\tau)}||z|^{\rho\alpha/s+\gamma}|dz|\\
		\leq& C\tau^{\gamma}(r_{1}+\tau)^{-\gamma-\rho\alpha/s-1}.
	\end{aligned}
\end{equation*}
	Collecting the above estimates and using the discrete Gr\"onwall inequality, the desired result is reached.
\end{proof}

\section{Numerical experiments}
In this section, we provide some numerical examples to verify the theoretical results. 
Here we take $Q$'s eigenvalues $\Lambda_{k}=k^{m}$, $k=1,2,\cdots$. According to the assumption $\|A^{-\rho}Q^{1/2}\|_{\mathcal{L}_{2}}<\infty$ and Lemma \ref{thmeigenvalue}, we have that $\rho$ is approximately equal to $\frac{1+m}{4}d$.

In the numerical experiments, we consider the equation
\begin{equation}\label{equretosolnum}
	\left \{
	\begin{split}
		&\partial_{t} u+\!_0\partial^{1-\alpha}_tA^{s} u
		=\sin(u)+\dot{W}^{H}_{Q} \,~\qquad\quad {\rm in}\ D,\ t\in(0,T],\\
		&u(\cdot,0)=0 \,\qquad\qquad\qquad\qquad\qquad\qquad~~ {\rm in}\ D,\\
		&u=0 \qquad\qquad\qquad\qquad\qquad\qquad\qquad~ \ \, {\rm on}\ \partial D,\ t\in(0,T],
	\end{split}
	\right .
\end{equation}
where $D=(0,1)$ and $T=0.01$. We take $100$ trajectories to calculate the solution of Eq. $\eqref{equretosolnum}$. Since the exact solution of \eqref{equretosolnum} is unknown, the spatial errors and temporal errors can be measured by
\begin{equation*}
	\begin{aligned}
		&e_{N}=\left (\frac{1}{100}\sum_{i=1}^{100}\|u^{L}_{N}(\omega_{i})-u^{L}_{2N}(\omega_{i})\|^{2}_{\mathbb{H}}\right )^{1/2},\\
		&e_{\tau}=\left (\frac{1}{100}\sum_{i=1}^{100}\|u_{\tau}(\omega_{i})-u_{\tau/2}(\omega_{i})\|^{2}_{\mathbb{H}}\right )^{1/2},
	\end{aligned}
\end{equation*}
where $u^{L}_{N}(\omega_{i})$ ($u_{\tau}(\omega_{i})$) means the numerical solution of $u$ at time $t_L$ with $N$ spectral bases  (step size $\tau$) and sample $\omega_{j}$; we can respectively calculate the spatial and temporal convergence rates by
\begin{equation*}
	{\rm Rate}=\frac{\ln(e_{N}/e_{2N})}{\ln(2)},\quad {\rm Rate}=\frac{\ln(e_{\tau}/e_{\tau/2})}{\ln(2)}.
\end{equation*}

\begin{example}
	Here, we verify the temporal convergence rate of the presented scheme (\ref{eqfullscheme}). We choose $N=100$ and $m=0,\ -0.5,\ -1$ with different $\alpha$ and $s$. The corresponding results with $H=0.3,\ 0.5,\ 0.8$ are shown in Tables \ref{tab:timeH03}, \ref{tab:timeH05}, and \ref{tab:timeH08}, respectively, and the predicted convergence rates are presented in bracket of the last column. All the convergence rates agree with the estimates provided in Theorem \ref{thmtimeer}.
	\begin{table}[htbp]
		\caption{Temporal errors and convergence rates with $H=0.3$}
		\begin{tabular}{cccccccl}
			\hline\noalign{\smallskip}
			$m$& $(\alpha,s)\backslash T/\tau$ & 32 & 64 & 128 & 256 & Rate \\
			\noalign{\smallskip}\hline\noalign{\smallskip}
			0 & (0.3,0.7) & 4.392E-02 & 9.707E-03 & 8.370E-03 & 7.452E-03 & 0.186 (0.193) \\
			& (0.5,0.7) & 3.485E-02 & 3.121E-02 & 2.812E-02 & 2.544E-02 & 0.145 (0.121) \\
			\noalign{\smallskip}\hline\noalign{\smallskip}
			-0.5& (0.3,0.4) & 4.541E-02 & 7.035E-03 & 5.839E-03 & 4.501E-03 & 0.306 (0.206) \\
			& (0.3,0.5) & 4.743E-02 & 6.694E-03 & 5.364E-03 & 4.277E-03 & 0.312 (0.225) \\
			\noalign{\smallskip}\hline\noalign{\smallskip}
			-1 & (0.3,0.7) & 5.477E-02 & 4.150E-03 & 3.186E-03 & 2.439E-03 & 0.403 (0.3) \\
			& (0.5,0.7) & 5.477E-02 & 4.042E-03 & 3.045E-03 & 2.153E-03 & 0.472 (0.3) \\
			\noalign{\smallskip}\hline
		\end{tabular}
		\label{tab:timeH03}
	\end{table}
	
	\begin{table}[htbp]
		\caption{Temporal errors and convergence rates with $H=0.5$}
		\begin{tabular}{cccccccc}
			\hline\noalign{\smallskip}
			$m$& $(\alpha,s)\backslash T/\tau$  & 32 & 64 & 128 & 256 & Rate \\
			\noalign{\smallskip}\hline\noalign{\smallskip}
			0 & (0.5,0.7) & 5.669E-02 & 4.250E-03 & 3.378E-03 & 2.653E-03 & 0.343 (0.321) \\
			& (0.6,0.7) & 5.345E-02 & 6.729E-03 & 5.387E-03 & 4.109E-03 & 0.363 (0.286) \\
			\noalign{\smallskip}\hline\noalign{\smallskip}
			-0.5 & (0.3,0.5) & 6.519E-02 & 8.896E-04 & 6.239E-04 & 4.543E-04 & 0.482 (0.425) \\
			& (0.5,0.5) & 6.124E-02 & 1.436E-03 & 9.607E-04 & 6.371E-04 & 0.603 (0.375) \\
			\noalign{\smallskip}\hline\noalign{\smallskip}
			-1 & (0.3,0.4) & 7.071E-02 & 4.388E-04 & 2.913E-04 & 1.861E-04 & 0.609 (0.5) \\
			& (0.3,0.7) & 7.071E-02 & 5.488E-04 & 3.746E-04 & 2.441E-04 & 0.569 (0.5) \\
			\noalign{\smallskip}\hline
		\end{tabular}
		\label{tab:timeH05}
	\end{table}
	
	\begin{table}[htbp]
		\caption{Temporal errors and convergence rates with $H=0.8$}
		\begin{tabular}{ccccccc}
			\hline\noalign{\smallskip}
			$m$& $(\alpha,s)\backslash T/\tau$  & 32 & 64 & 128 & 256 &   Rate \\
			\noalign{\smallskip}\hline\noalign{\smallskip}
			0 & (0.5,0.7) & 7.883E-02 & 1.856E-04 & 1.189E-04 & 7.688E-05 & 0.646 (0.621)  \\
			& (0.6,0.7) & 7.653E-02 & 3.032E-04 & 1.929E-04 & 1.218E-04 & 0.665 (0.586)   \\
			\noalign{\smallskip}\hline\noalign{\smallskip}
			-0.5 	& (0.3,0.5) & 8.515E-02 & 4.132E-05 & 2.403E-05 & 1.330E-05 & 0.804 (0.725)   \\
			& (0.5,0.5) & 8.216E-02 & 6.741E-05 & 3.781E-05 & 2.072E-05 & 0.863 (0.675)   \\
			\noalign{\smallskip}\hline\noalign{\smallskip}
			-1 & (0.3,0.7) & 8.944E-02 & 2.742E-05 & 1.405E-05 & 8.038E-06 & 0.876 (0.8)   \\
			& (0.5,0.5) & 8.944E-02 & 2.480E-05 & 1.272E-05 & 6.399E-06 & 0.970 (0.8)   \\
			\noalign{\smallskip}\hline
		\end{tabular}
		\label{tab:timeH08}
	\end{table}

\end{example}

\begin{example}
	Spatial convergence rate of the scheme (\ref{eqfullscheme}) is validated in this example. Here we take $\tau=T/2048$ with  different $\alpha$ and $s$. We choose $m=0,\ -0.5,\ -1$ and the corresponding errors and convergence rates with $H=0.3,\ 0.5,\ 0.8$ are presented in Tables \ref{tab:spaceH03}, \ref{tab:spaceH05}, and \ref{tab:spaceH08}, respectively. All the convergence rates are consistent with the predicted ones  in Theorem \ref{thmspter} (which are presented in bracket of the last column).
	\begin{table}[htbp]
		\caption{Spatial errors and convergence rates with $H=0.3$}
		\begin{tabular}{ccccccc}
			\hline\noalign{\smallskip}
			$m$& $(\alpha,s)\backslash N$  & 8 & 16 & 32 & 64 & Rate \\
		\noalign{\smallskip}	\hline\noalign{\smallskip}
			0 & (0.3,0.7) & 9.487E-02 & 4.303E-02 & 2.849E-02 & 1.596E-02 & 0.743 (0.9) \\
			& (0.6,0.7) & 4.472E-02 & 1.587E-01 & 1.511E-01 & 1.273E-01 & 0.173 (0.2) \\
			\noalign{\smallskip}\hline\noalign{\smallskip}
			-0.5 & (0.3,0.4) & 7.416E-02 & 6.858E-02 & 4.717E-02 & 3.210E-02 & 0.555 (0.55) \\
			& (0.3,0.7) & 1.072E-01 & 1.984E-02 & 8.381E-03 & 3.588E-03 & 1.261 (1.15) \\
			\noalign{\smallskip}\hline\noalign{\smallskip}
			-1 & (0.3,0.4) & 8.944E-02 & 2.689E-02 & 1.341E-02 & 6.787E-03 & 1.014 (0.8) \\
			& (0.3,0.7) & 1.183E-01 & 7.179E-03 & 2.437E-03 & 8.080E-04 & 1.667 (1.4) \\
			\noalign{\smallskip}\hline
		\end{tabular}
		\label{tab:spaceH03}
	\end{table}
	
	\begin{table}[htbp]
		\caption{Spatial errors and convergence rates with $H=0.5$}
		\begin{tabular}{ccccccc}
			\hline\noalign{\smallskip}
			$m$& $(\alpha,s)\backslash N$  & 8 & 16 & 32 & 64 &Rate  \\
			\noalign{\smallskip}\hline\noalign{\smallskip}
			0 & (0.3,0.7) & 9.487E-02 & 1.308E-02 & 8.560E-03 & 4.801E-03 & 0.780 (0.900) \\
			& (0.6,0.7) & 8.165E-02 & 5.509E-02 & 3.973E-02 & 2.787E-02 & 0.518 (0.667) \\
			\noalign{\smallskip}\hline\noalign{\smallskip}
			-0.5 & (0.3,0.4) & 7.416E-02 & 2.416E-02 & 1.704E-02 & 1.143E-02 & 0.612 (0.550) \\
			& (0.6,0.4) & 6.455E-02 & 5.299E-02 & 4.490E-02 & 3.519E-02 & 0.351 (0.417) \\
			\noalign{\smallskip}\hline\noalign{\smallskip}
			-1 & (0.3,0.4) & 8.944E-02 & 9.550E-03 & 5.066E-03 & 2.397E-03 & 1.062 (0.800) \\
			& (0.6,0.4) & 8.165E-02 & 2.283E-02 & 1.371E-02 & 7.445E-03 & 0.857 (0.667) \\
			\noalign{\smallskip}\hline
		\end{tabular}
		\label{tab:spaceH05}
	\end{table}

	\begin{table}[htbp]
		\caption{Spatial errors and convergence rates with $H=0.8$}
		\begin{tabular}{ccccccc}
			\hline\noalign{\smallskip}
			$m$& $(\alpha,s)\backslash N$  & 8 & 16 & 32 & 64 & Rate \\
			\noalign{\smallskip}\hline\noalign{\smallskip}
			0 & (0.3,0.4) & 5.477E-02 & 1.493E-02 & 1.400E-02 & 1.261E-02 & 0.155 (0.3) \\
			& (0.3,0.7) & 9.487E-02 & 3.258E-03 & 2.028E-03 & 1.102E-03 & 0.803 (0.9) \\
			\noalign{\smallskip}\hline \noalign{\smallskip}
			-0.5 & (0.3,0.4) & 7.416E-02 & 5.713E-03 & 4.251E-03 & 2.609E-03 & 0.595 (0.55) \\
			& (0.6,0.7) & 1.072E-01 & 5.062E-03 & 2.350E-03 & 9.963E-04 & 1.207 (1.15) \\
			\noalign{\smallskip}\hline \noalign{\smallskip}
			-1 & (0.6,0.4) & 8.944E-02 & 5.312E-03 & 3.246E-03 & 1.730E-03 & 0.862 (0.8) \\
			& (0.6,0.7) & 1.183E-01 & 1.963E-03 & 7.343E-04 & 2.278E-04 & 1.646 (1.4) \\
			\noalign{\smallskip}\hline
		\end{tabular}
		\label{tab:spaceH08}
	\end{table}
	
\end{example}

\section{Conclusions}

We offer the unified numerical analysis for stochastic nonlinear fractional diffusion equation driven by fractional Gaussian noise with Hurst index $H\in (0,1)$. The regularity estimates of mild solution in time and space are developed based on a novel estimate of the second moment of stochastic integral of fBm. The fully discrete scheme constructed by spectral Galerkin method and backward Euler convolution quadrature is proposed and optimal error estimates are obtained. The theoretical results are also verified by numerical experiments.

%


%
%

\bibliographystyle{spmpsci}
\bibliography{refH05}

\begin{thebibliography}{10}
\providecommand{\url}[1]{{#1}}
\providecommand{\urlprefix}{URL }
\expandafter\ifx\csname urlstyle\endcsname\relax
  \providecommand{\doi}[1]{DOI~\discretionary{}{}{}#1}\else
  \providecommand{\doi}{DOI~\discretionary{}{}{}\begingroup
  \urlstyle{rm}\Url}\fi

\bibitem{Acosta.2019Feaffep}
Acosta, G., Bersetche, F.M., Borthagaray, J.P.: {Finite element approximations
  for fractional evolution problems}.
\newblock Fract. Calc. Appl. Anal. \textbf{22}, 767--794 (2019)

\bibitem{Acosta.2017AFLERoSaFEA}
Acosta, G., Borthagaray, J.P.: {A fractional Laplace equation: regularity of
  solutions and finite element approximations}.
\newblock SIAM J. Numer. Anal. \textbf{55}, 472--495 (2017)

\bibitem{Arezoomandan.2021ScmfspdewfBm}
Arezoomandan, M., Soheili, A.R.: {Spectral collocation method for stochastic
  partial differential equations with fractional Brownian motion}.
\newblock J. Comput. Appl. Math. \textbf{389}, 113369 (2021)

\bibitem{Banna.2019Fbm}
Banna, O.: {Fractional Brownian Motion: Approximations and Projections}.
\newblock {John Wiley and Sons Inc}, Hoboken NJ (2019)

\bibitem{Bardina.2006MfiwHplt12}
Bardina, X., Jolis, M.: {Multiple fractional integral with Hurst parameter less
  than 1/2}.
\newblock Stochastic Process. Appl. \textbf{116}, 463--479 (2006)

\bibitem{Cao.2017ASEEwAWaRN}
Cao, Y., Hong, J., Liu, Z.: {Approximating stochastic evolution equations with
  additive white and rough noises}.
\newblock SIAM J. Numer. Anal. \textbf{55}, 1958--1981 (2017)

\bibitem{Cao.2018FeafsosdedbfBm}
Cao, Y., Hong, J., Liu, Z.: {Finite element approximations for second-order
  stochastic differential equation driven by fractional Brownian motion}.
\newblock IMA J. Numer. Anal. \textbf{38}, 184--197 (2018)

\bibitem{DiNezza.2012HgttfSs}
{Di Nezza}, E., Palatucci, G., Valdinoci, E.: {Hitchhiker's guide to the
  fractional Sobolev spaces}.
\newblock Bull. Sci. Math. \textbf{136}, 521--573 (2012)

\bibitem{Ervin.2006Vfftsfade}
Ervin, V.J., Roop, J.P.: {Variational formulation for the stationary fractional
  advection dispersion equation}.
\newblock Numer. Methods Partial Differential Equations \textbf{22}, 558--576
  (2006)

\bibitem{Gunzburger.2018ScrotdfstfPstastwn}
Gunzburger, M., Li, B., Wang, J.: {Sharp convergence rates of time
  discretization for stochastic time-fractional PDEs subject to additive
  space-time white noise}.
\newblock Math. Comp. \textbf{88}, 1715--1741 (2018)

\bibitem{Jin.2020Iisos}
Jin, B., Zhou, Z.: {Incomplete iterative solution of subdiffusion}.
\newblock Numer. Math. \textbf{145}, 693--725 (2020)

\bibitem{Kloeden.1992Nsosde}
Kloeden, P.E., Platen, E.: {Numerical Solution of Stochastic Differential
  Equations}.
\newblock Springer-Verlag, Berlin and New York (1992)

\bibitem{Laptev.1997DaNEPoDiES}
Laptev, A.: {Dirichlet and Neumann eigenvalue problems on domains in Euclidean
  spaces}.
\newblock J. Funct. Anal. \textbf{151}, 531--545 (1997)

\bibitem{Li.1983OtSeatep}
Li, P., Yau, S.T.: {On the Schr\"odinger equation and the eigenvalue problem}.
\newblock Comm. Math. Phys. \textbf{88}, 309--318 (1983)

\bibitem{Li.2017GFEAfSSTFWE}
Li, Y., Wang, Y., Deng, W.: {Galerkin finite element approximations for
  stochastic space-time fractional wave equations}.
\newblock SIAM J. Numer. Anal. \textbf{55}, 3173--3202 (2017)

\bibitem{Liu.2021HOAfSSFWEFbaASTGN}
Liu, X., Deng, W.: {Higher order approximation for stochastic space fractional
  wave equation forced by an additive space-time Gaussian noise}.
\newblock J. Sci. Comput. \textbf{87}, 11 (2021)

\bibitem{Lubich.1988CqadocIb}
Lubich, C.: {Convolution quadrature and discretized operational calculus. I}.
\newblock Numer. Math. \textbf{52}, 129--145 (1988)

\bibitem{Lubich.1988CqadocI}
Lubich, C.: {Convolution quadrature and discretized operational calculus. II}.
\newblock Numer. Math. \textbf{52}, 413--425 (1988)

\bibitem{Lubich.1996Ndeefaoaeewaptmt}
Lubich, C., Sloan, I.H., Thom{\'e}e, V.: {Nonsmooth data error estimates for
  approximations of an evolution equation with a positive-type memory term}.
\newblock Math. Comp. \textbf{65}, 1--18 (1996)

\bibitem{Mandelbrot.1968FBMFNaA}
Mandelbrot, B.B., {van Ness}, J.W.: {Fractional Brownian motions, fractional
  noises and applications}.
\newblock SIAM Rev. \textbf{10}, 422--437 (1968)

\bibitem{Mishura.2008ScffBmarp}
Mishura, I.S.: {Stochastic Calculus for Fractional Brownian Motion and Related
  Processes}.
\newblock Springer, Berlin (2008)

\bibitem{Podlubny.1999Fde}
Podlubny, I.: {Fractional Differential Equations}.
\newblock Academic, San Diego and London (1999)

\bibitem{Simonsen.2003Macitnesmbw}
Simonsen, I.: {Measuring anti-correlations in the nordic electricity spot
  market by wavelets}.
\newblock Phys. A \textbf{322}, 597--606 (2003)

\bibitem{Song.2003Vondra}
Song, R., Vondra\u{c}ek, Z.: {Potential theory of subordinate killed Brownian
  motion in a domain}.
\newblock Probab. Theory Relat. Fields. \textbf{125}, 578--592 (2003)

\bibitem{Wang.2017SmsrrfSwfnaocrftna}
Wang, X., Qi, R., Jiang, F.: {Sharp mean-square regularity results for SPDEs
  with fractional noise and optimal convergence rates for the numerical
  approximations}.
\newblock BIT \textbf{57}, 557--585 (2017)

\bibitem{Wu.2020AaotLsfsspdbistwn}
Wu, X., Yan, Y., Yan, Y.: {An analysis of the L1 scheme for stochastic
  subdiffusion problem driven by integrated space-time white noise}.
\newblock Appl. Numer. Math. \textbf{157}, 69--87 (2020)

\bibitem{Yan.2019OeeffspdewfBm}
Yan, L., Yin, X.: {Optimal error estimates for fractional stochastic partial
  differential equation with fractional Brownian motion}.
\newblock Discrete Contin. Dyn. Syst. Ser. B \textbf{24}, 615--635 (2019)

\end{thebibliography}

\end{document}